\documentclass[11pt,letter]{amsart}
\usepackage[letterpaper,left=1in,top=1in,right=1in,bottom=1in]{geometry}
\usepackage{mathrsfs, amssymb,amsmath,url,amsthm,color,comment}
\usepackage{hyperref}
\usepackage[dvipsnames]{xcolor}

\usepackage{multirow}

\usepackage{tikz}
\usetikzlibrary{graphs,calc}

\title[Permutation groups on countable vector spaces over prime fields]{Permutation groups on countable vector spaces\\ over prime fields}

\author [B. Bodor]{Bertalan Bodor}
	\address{Institut f\"{u}r Algebra, TU Dresden, Germany}
	\email{bertalan.bodor@tu-dresden.de}

\author [M. Pinsker]{Michael Pinsker}
	\address{Institut f\"{u}r Diskrete Mathematik und Geometrie, FG Algebra, TU Wien, Austria, and Department of Algebra, Charles University, Czech Republic}    
	\email{marula@gmx.at}
    \urladdr{http://dmg.tuwien.ac.at/pinsker/}
\thanks{Michael Pinsker has received funding from the  Austrian Science Fund (FWF) through  project No P32337, and from the Czech Science Foundation (grant No 18-20123S)}

\author [L. Schiffer]{Lyra Schiffer}
	\address{Institut f\"{u}r Diskrete Mathematik und Geometrie, FG Algebra, TU Wien, Austria}
	\email{lyra\_schiffer@outlook.com}

\author [C. Szab\'{o}]{Csaba Szab\'{o}}
	\address{Algebra \'{e}s Sz\'{a}melm\'{e}let Tansz\'{e}k, 
E\"{o}tv\"{o}s Lor\'{a}nd Tudom\'{a}nyegyetem, Budapest, 
Hungary}
	\email{csaba@cs.elte.hu}
	\urladdr{http://web.cs.elte.hu/~csaba/index.html}
\thanks{Csaba Szab\'{o} has received funding from the NRDI Fund  under the FK 18 funding scheme (NKFI-128673)}

\usepackage{t1enc}
\usepackage[utf8]{inputenc}
\usepackage{ wasysym }
\usepackage{verbatim}
\usepackage{cancel}
\usepackage{enumitem}
\usepackage[colorinlistoftodos]{todonotes}

\usepackage{hyperref}

\newcommand\purple[1]{{#1}}

\DeclareMathOperator{\Aut}{{Aut}}
\DeclareMathOperator{\Sym}{{Sym}}
\DeclareMathOperator{\Aff}{{Aff}}
\DeclareMathOperator{\acl}{{acl}}

\DeclareMathOperator{\id}{{id}}

\DeclareMathOperator{\SemL}{\Gamma {L}}
\DeclareMathOperator{\0}{\mathbf{0}}

\DeclareMathOperator{\AGL}{AGL}

\newcommand{\ignore}[1]{}

\newcommand{\AutF}[1]{\Aut_B (#1)}
\newcommand{\oneD}[1]{\mathrm{S}_1(#1)}

\newcommand{\Sstar}{\Sym^\ast}
\newcommand{\SG}[2]{S(#1,#2)}

\newcommand{\linClos}[1]{\left\langle #1 \right\rangle}

\newcommand{\V}{\mathcal{V}}

\newcommand{\Vtimes}{V\setminus \{\0\}}

\newcommand{\F}[1]{\mathbb{F}_{#1}}

\newcommand{\Fi}{\mathcal{F}}
\newcommand{\A}{\mathcal{A}}
\newcommand{\B}{\mathcal{B}}
\newcommand{\Gr}{\mathcal{G}}
\newcommand{\W}{\mathcal{W}}
\newcommand{\N}{\mathbb{N}}
\newcommand{\Z}{\mathbb{Z}}

\newcommand{\M}{\mathcal{M}}

\newcommand{\ut}{\bar{u}}
\newcommand{\uh}{\hat{u}}
\newcommand{\gt}{\bar{\gamma}}
\newcommand{\gh}{\hat{\gamma}}
\newcommand{\gtt}{\bar{g}}
\newcommand{\ghh}{\hat{g}}
\newcommand{\htt}{\bar{h}}
\newcommand{\hhh}{\hat{h}}
\newcommand{\stt}{\bar{\sigma}}
\newcommand{\shh}{\hat{\sigma}}
\newcommand{\trs}[1]{\tau_{#1}}

\newcommand{\presprojlines}{preserves projective lines}

\newcommand{\tilg}{\tilde{g}}

\newcommand{\otr}[3]{
\left(
\begin{array}{c}
     #1  \\
     #2  \\
     #3
\end{array} \right)
}

\newcommand{\Gast}{\mathcal{G}^\ast}
\newcommand{\G}{\mathcal{G}}
\newcommand{\Hc}{\mathcal{H}}
\newcommand{\Nr}{\mathcal{N}}
\newcommand{\Sr}{\mathcal{S}}
\newcommand{\K}{\mathcal{K}}
\newcommand{\Mr}{\mathcal{M}}

\theoremstyle{plain}
\newtheorem{theorem}{Theorem}[section]
\newtheorem{lemma}[theorem]{Lemma}
\newtheorem{proposition}[theorem]{Proposition}
\newtheorem{corollary}[theorem]{Corollary}

\theoremstyle{definition}
\newtheorem{definition}[theorem]{Definition}

\newcommand{\ch}{\operatorname{char}}

\begin{document}

\begin{abstract}
We describe all closed permutation groups which act on the set of vectors of a countable vector space $\V$ over a prime field of odd order and which contain all automorphisms of $\V$. In particular, we  prove that their number is finite. These groups correspond, up to first-order interdefinability,  precisely to all structures with a first-order definition in $\V$. 
\end{abstract} \maketitle
\section{Introduction}\label{sect:intro}

When dealing with a structure $\A$, we are often interested in other structures which can be defined in $\A$, but do not retain the full expressiveness of the original structure $\A$. For example, when investigating the affine geometry of a vector space, we consider the affine structure of the vector space, which  is, of course, roughly  the vector space structure without the distinguished  role of the zero vector. We call structures with a first-order definition in $\A$ \emph{first-order reducts} of $\A$; in the case of the affine structure of a vector space, this structure  can be represented by a single 4-ary relation which relates four vectors if the first is an affine combination of the others. If we distinguish different structures by their automorphisms, then giving up some information on $\A$ amounts to more possible transformations; in the example of vector spaces and their affine structure, it means adding the translations to the group of automorphisms of the vector space. If $\A$ is a countable $\omega$-categorical structure, i.e., the up to isomorphism unique countable model of its first-order theory, then two first-order reducts of $\A$ are interdefinable (i.e., first-order reducts of one another) if and only if they cannot be distinguished by automorphisms. In fact, the first-order reducts of $\A$, up to interdefinability, correspond precisely to the permutation groups acting on $\A$ which contain all automorphisms of $\A$ and which are closed in the topology of pointwise convergence. Any countable vector space over a finite field is $\omega$-categorical, and an example of a closed permutation group containing its automorphisms is the general affine group acting on the vectors.

Vector spaces are homogeneous in the sense that any isomorphism between finitely generated subspaces extends to an automorphism of the entire space. Other examples of structures which are homogeneous in this sense are the countable atomless Boolean algebra, the order of the rationals, or the random graph. It has been conjectured that if $\A$ is a countable homogeneous structure in a finite relational signature, then there are only finitely many first-order reducts of $\A$ up to interdefinability~\cite{RandomReducts}. Since any such structure is $\omega$-categorical, this is equivalent to there being only a finite number of closed permutation groups containing the automorphism group of $\A$. This conjecture has been verified for numerous structures including the order of the rationals~\cite{Cameron5} and the random graph~\cite{RandomReducts}, but is still open~\cite{Thomas96,Bennett-thesis,JunkerZiegler,BP-reductsRamsey,Poset-Reducts,LinmanPinsker,42,BodJonsPham,agarwal,Pon11,AgarwalKompatscher,BBPP18,bodirsky2021permutation,bodor2020classification}. For countable homogeneous structures in a finite  non-relational signature, it turns out to be false: the countable vector space over the two-element field, with an additionally distinguished non-zero vector, has an infinite number of first-order reducts~\cite{BodorCameronSzabo}, although it has only finitely many (four) first-order reducts if it is not equipped with any additional structure beyond the vector space structure~\cite{BoKalSz}.

In this article we determine all closed supergroups containing the automorphism group of a countable vector space $\V$ over a prime field of odd order; and hence its first-order reducts up to interdefinability. As it turns out, the affine group and the full symmetric group are the only such groups which do not fix the zero vector of $\V$. The groups which do fix the zero vector fall into two classes: those which act on the projective space of $\V$, and those which do not; the former can be further divided into those which preserve projective lines and those who do not. In all cases, the groups are fully determined by a pair of subgroups of the symmetric group on the underlying field, one of which is normal. Since the field is finite, there are only finitely many possibilities for such pairs.

\begin{theorem}\label{thm:main}
Let $\V$ be a countable vector space over a prime field of odd order. Then there is a finite number of closed permutation groups containing the automorphism group of $\V$.
\end{theorem}

We begin with a section in which we establish our notation and provide some fundamental facts on vector spaces and $\omega$-categoricity (Section~\ref{sect:prelims}). The largest part of our proof is given in Section~\ref{sect:fix0} dealing with groups which fix the zero vector. In Section~\ref{sect:move0} we finish the proof by considering those groups which move the zero vector. A summary of the results and the proof is provided in Figure~\ref{fig:summary} at the end of this article. \section{Preliminaries}\label{sect:prelims}

\subsection{Relations and functions}
{ Let $S$ be a set and let $n \geq 1$.} The cardinality of $S$ is denoted by $|S|$.  A subset $R$ of $S^n$ is  an  \emph{$n$-ary relation $R$} of $S$.  If $R$ is a binary relation on $S$ and $(a,b)\in S^2$, we write  $a \ R \ b$ to indicate  $(a,b)\in R$. Let $\sim$ be an equivalence relation on $S$. 
For an element $s\in S$ the {equivalence class of $s$} is denoted  by $[s]_\sim$, and any element in $[s]_\sim$ is called a \emph{representative of} $[s]_\sim$. { We define for any subset $S'$ of $S$ the set $[S']_\sim:=\bigcup_{s\in S'} [s]_\sim$.}  An $n$-ary \emph{operation} on $S$ is a function  from  $S^n$  into the set $S$. 

A \emph{(first-order) structure} consists  of a set, called the \emph{domain if the structure}, and relations and functions thereon. 
We will, if not otherwise specified, denote structures by a letter in calligraphic font, such as $\A$, and the corresponding domain by the same letter in plain font, such as $A$. A structure without functions is called a \emph{relational structure}. 
A structure $\mathcal{B}$ on the same domain as $\mathcal A$ is a \emph{first-order reduct} of $\mathcal{A}$ iff all operations  and all relations of $\mathcal{B}$ are first-order definable over $\A$. Since in this article we are only interested in {first-order} reducts we shall often omit  ``first-order''  and  write that $\B$ is a {reduct} of $\A$.
Two structures are called \emph{(first-order) interdefinable} if they are first-order reducts of one another.

Let $f\colon A \to B$ and $g\colon B \to C$ be functions. We usually write $a^f$ for $f(a)$ ($a\in A$) and $fg$ for the composition $g\circ f\colon a\mapsto g(f(a))$.
For a set $S\subseteq A$  we define $S^f:=\{s^f:s\in S\}$. We denote the {restriction of $f$ to $S$} by $f|_S\colon S\to B$. In this case $f$ is an \emph{extension of $f|_S$}. The {identity function on $S$} is denoted by $\text{id}_S\colon S \to S$.

\subsection{Groups}
Let $\Gr$ be a group. We write $\Hc\leq \Gr$ if $\Hc$ is a \emph{subgroup of $\Gr$}. If a subgroup $\Nr$ is normal in $\Gr$ we denote this by $\Nr\lhd \Gr$. For $\Hc\leq \Gr$ and $\Nr\lhd \Gr$ the group $\Gr$ is the \emph{semi-direct product} $\Hc\ltimes \Nr$ iff $G=HN=\{hn:h\in H, n\in N\}$ and $H\cap N=\{0\}$.

The group $\Sym(A)$ of all  bijective functions on a set $A$ is  \emph{the full symmetric group on $A$}. A group $\Gr$ is called a \emph{permutation group if $\Gr\leq \Sym(A)$ for some set $A$.}
For a permutation group $\Gr\leq\Sym(A)$ and a subset $S$ of $A$ the subgroup of $\Gr$ consisting of all functions of $\G$ which fix $S$ element-wise is called the \emph{stabilizer of $S$} and is denoted by $\Gr_S$. If $S$ consists only of finitely many elements $x_1,\ldots,x_n\in A$ we write $\Gr_{x_1,\ldots,x_n}$ instead of $\Gr_{\{x_1,\ldots,x_n\}}$.  
For a function $f\in \Sym(A)$ and sets of functions $S, T\subseteq \Sym(A)$ we  write $fS:= \{fg: g\in S\}$, $Sf:= \{gf: g\in S\}$, and $ST:= \{gh: g\in S,h\in T\}$. 

A subset $S$ of $\Sym(A)$ is \emph{closed with respect to the pointwise convergence topology} iff for all $f\in \Sym(A)$ such that for every finite subset $F\subseteq A$ there exists a function $g\in S$ such that $g|_F=f|_F$ we have $f\in S$. A sequence $(f_n)_{n\geq 0}$ of functions in $\Sym(A)$ is \emph{convergent} if there exists a function  $f\in \Sym(A)$ such that for all $a\in A$   exists  $N\geq 0$ such that for all $n>N\colon f_n(a)=f(a)$. 
The function $f$ is the \emph{limit of $(f_n)_{n\geq 0}$}. 
Using this  we  have  an equivalent definition of closedness: a subset $S$ of $\Sym(A)$ is closed if and only if the limit of every convergent sequence of functions in $S$  lies in $S$. {A permutation group $\G\leq \Sym(A)$ is \emph{closed} if the set of its functions are closed in $\Sym(A)$; this is the case if and only if it is the automorphism group of a structure with domain $A$.}

Let $A$ be a set, let $\G$ be a permutation group on $A$ and let $s=(s_1,\ldots,s_n)\in A^n$ for some $n\geq 1$. Then the \emph{orbit of $s$} is  the set $G(s):=\{s^g:g\in G\}$, where $s^g:=(s_1^g,\ldots,s_n^g)$. The group $\G$ is \emph{oligomorphic} iff for every $n\geq 1$ the set of orbits of $n$-tuples in $S^n$ is finite. We call countable structures with an oligomorphic automorphism group \emph{$\omega$-categorical}; for countable structures this definition coincides with the standard model-theoretical definition (see for instance~\cite{Hodges}). Note that the countably infinite vector space $\V$ over any finite field $\Fi$ is $\omega$-categorical, by its homogeneity and since any $n$-tuple of vectors is contained in an $n$-dimensional subspace of $\V$, which has precisely $|F|^n$ elements. If $\A$ is a first-order reduct of $\V$, then $\Aut(\A)\supseteq\Aut(\V)$, and thus $\A$ is $\omega$-categorical as well. In general, for countable $\omega$-categorical structures $\A, \B$ on the same domain, $\Aut(\B)\supseteq \Aut(\A)$ if and only if $\B$ is a reduct of $\A$; in particular $\Aut(\A)=\Aut(\B)$ if and only if $\A$ and $\B$ are interdefinable. Hence, the reducts of $\V$, up to interdefinability, correspond precisely to the closed permutation groups containing $\Aut(\V)$.

\subsection{Fields}
The addition in a field $\Fi$ shall be denoted by $+$, and the multiplication by $\cdot$;  the neutral element of the addition will bear the symbol $0$, and  the neutral element of the multiplicative group with domain $F^\times:=F\setminus \{0\}$ the symbol $1$. For any two elements $a\in F$ and $b \in F^\times$ we write $\frac{a}{b}:=a\cdot b^{-1}$, where $b^{-1}$ denotes the multiplicative inverse of $b$. For any prime number $p$ and $n\geq 1$, we denote the unique 
field with $p^n$ elements by $\F{p^n}$.

\subsection{Vector Spaces}
Let $\V$ be a {vector space}  over a field $\Fi$. The symbol $\0$ denotes the  {zero vector} of $\V$, and for every $c\in F$ and every $v\in V$ we denote by $cv$ the product of $v$ with the scalar $c$. We write $-a$ for the additive inverse of any vector, and $a-b$ is short for $a+(-b)$. 
   The \emph{linear closure}  of a set $S\subseteq V$ is  denoted by $\linClos{S}$. { For finitely many vectors $v_1,\ldots,v_n \in V$ we write $\linClos{v_1,\ldots,v_n}$ instead of $\linClos{\{v_1,\ldots,v_n\}}$.} A subset of $V$  is \emph{linearly independent} if it is not contained in the linear closure of any proper subset, and a tuple of elements of $V$ is linearly independent if none of its component is contained in the linear closure of the other components. If $S=\linClos{S}$, i.e., if $S$ is a subspace of $\V$, then we write $S\leq \V$. We denote the set of all subspaces of $\V$ by $\text{S}(\V)$ and the set of all one-dimensional subspaces of $\V$ by $\oneD{\V}$.  The \emph{dimension} of a subspace $W\leq \V$ is denoted by $\dim W$. We define the dimension of a tuple $t$ as the dimension of the subspace spanned by its components. Furthermore, abusing notation we write $\linClos{t}$ for the linear closure of the components of the tuple $t$, and set $\dim t:= \dim \linClos{t}$. 
 For two subspaces $W,U\leq \V$ we denote the sum space by $W+U$, and similarly, we write $\sum_{i\in I}W_i$ for the sum space of an arbitrary family of subspaces $(W_i)_{i\in I}$ of $\V$.

An \emph{affine subspace} of a vector space $\V$ is any subset $A$ of $V$ such that for any $v\in A$ the set $A-v=\{a-v\in V:a\in A\}$ is a subspace of $\V$. { The \emph{dimension} of an affine subspace of $\V$ is the dimension of the corresponding subspace of $\V$. {For a set of vectors $S\subseteq V$ we write $\Aff(S)$ for the \emph{affine closure} of $S$ in $V$ which is the set of all  \emph{affine combinations}, i.e., all  $\sum_{s\in S} c_s s$ such that almost all $c_s\in F$ are zero and the sum $\sum_{s\in S} c_s$ equals $1\in F$. For finitely many $v_1,\ldots,v_n\in V$ we write $\Aff(v_1,\ldots,v_n)$ instead of $\Aff(\{v_1,\ldots,v_n\})$. } We call any one-dimensional affine space, i.e., a set of the form $\Aff(v,w)$ or $v+\linClos{w-v}$ for some distinct vectors $v,w\in V$, an \emph{affine line}.
}
An \emph{affine mapping} between affine subspaces is a mapping which preserves affine combinations. The set of all bijective affine mappings from $V$ to $V$ is denoted by $\AGL(\V)$. 

	We now fix the structure we consider in this article.\\
\begin{center}\fbox{
\begin{minipage}{.9\linewidth}
\purple{Unless we explicitly state otherwise $p$ denotes a fixed odd prime and $\V$ denotes a countably infinite dimensional vector space over the field $\F{p}$.}
\end{minipage}
}
\end{center}

\section{The closed supergroups of \texorpdfstring{$\Aut(\V)$}{Aut(V)}  fixing \texorpdfstring{$\0$}{0}}\label{FixingZero}\label{sect:fix0}
Throughout this section we denote by $\G$ a closed permutation group on $V$ such that 
\begin{align*}
    \Aut (\V) \leq \G \leq \Sym (V)_{\0}.
\end{align*}
\purple{Our goal in this section is to give a full description of the group $\G$.} We start by examining the action of $\G$ on the set of equivalence classes of a suitable equivalence relation.

\begin{definition}
Let $\Gamma$ be a multiplicative subgroup of $\mathbb{F}_{p}^\times$. We define an equivalence relation $\sim_\Gamma$ on $V$ by setting for all $v,w\in V$:
\begin{align*}
v \sim_\Gamma w     \text{ iff  }\text{ for some } \lambda \in  \Gamma\colon v=\lambda w.
\end{align*}
\end{definition}
We note that for all $\Gamma\leq \F{p}^\times$:
\begin{itemize}
	\item $\sim_\Gamma$ is an equivalence relation since $\Gamma$ is a group.
	\item For all $v\in V$ we have $[v]_{\sim_\Gamma}\subseteq \langle v \rangle$. 
	\item The equivalence class of $\0$ is $\{\0\}$.
\end{itemize} 
For $\G$ to act on the equivalence classes of $\sim_\Gamma$ it will be necessary  that $\sim_\Gamma$ be \emph{$\G$-invariant}, i.e., for every $g\in G$ and for all $v,w\in V$ we have that $v\sim_\Gamma w$ implies $v^g \sim_\Gamma w^g$.  The subgroup $\Gamma$ of $\F{p}^\times$ needs to be of a specific form for this to be the case. We now  define  another equivalence relation $\sim_{\G}$ which will obviously be $\G$-invariant and then show that there is a group $\Gamma\leq \F{p}^\times$ such that $\sim_{\G}=\sim_{\Gamma}$.

\begin{definition}\label{equiv_G}
We define an equivalence relation $\sim_{\G}$ on $V$ by setting for all vectors $v,w \in V$:
\begin{align*}
v\sim_{\G} w \text{ iff  for all } g\in G\colon  \langle v^g \rangle=\langle w^g \rangle.
\end{align*}
\end{definition}
We note that: 
\begin{itemize}
    \item $\sim_{\G}$ is an equivalence relation, and it is  invariant under $\G$ since $\G$ is a group.
    \item For all $v\in V$ we have $[v]_{\sim_{\G}}\subseteq \langle v \rangle$.
\item The equivalence class of $\0$ is  $\{\0\}$.
\end{itemize}

\begin{lemma}\label{equiv_is_equiv}
There exists a group $\Gamma \leq \mathbb{F}_{p}^\times$ such that ${\sim_{\G}}={\sim_\Gamma}$.
\end{lemma}
\begin{proof}
We fix an arbitrary vector $v\neq \0$ and consider the set
\begin{align*}
    M_v:=\{\lambda \in  \mathbb{F}_{p}^\times: v\sim_{\G} \lambda v\}.
\end{align*}
We claim that $M_v$ does not depend on $v$. This holds, as for any other vector $w\neq \0$ there exists an automorphism $\varphi$ of $\V$ which maps $v$ to $w$. Therefore, $\lambda v \sim_{\G} v$ iff $\lambda w \sim_{\G} w$ because of the $\G$-invariance of $\sim_{\G}$. 

Observe that $\Gamma:=M_v$ contains $1\in \F{p}$ because $\sim_{\G}$ is reflexive. For any $\lambda \in \Gamma$ its inverse $\lambda^{-1}$ is also an element of $\Gamma$ since $\sim_{\G}$ is symmetric. Furthermore, $\Gamma$ is closed under multiplication due to $\sim_{\G}$ being transitive. This shows that $\Gamma$ is a subgroup of $\mathbb{F}_{p}^\times$. For all elements $\lambda \in \mathbb{F}_{p}^\times$ and for all vectors $v\in V\setminus\{\0\}$ we have
\begin{align*}
    v \sim_\Gamma \lambda v \iff \lambda \in \Gamma \iff v \sim_{\G} \lambda v.
\end{align*}

{Any vector outside of $\langle v \rangle$ cannot be equivalent to $v$ under $\sim_{\G}$ or $\sim_\Gamma$, and we have already seen $[\0]_{\sim_{\G}}=[\0]_{\sim_\Gamma}=\{\0\}$.} {This shows that  ${\sim_{\G}}$ and ${\sim_\Gamma}$ coincide.}
\end{proof}

	For the rest of Chapter~\ref{FixingZero} we fix a group $\Gamma\leq \F{p}^\times$ as in the conclusion of Lemma~\ref{equiv_is_equiv}. We will denote the equivalence relation ${\sim_{\G}}={\sim_\Gamma}$ by $\sim$.


In the case $\Gamma=\F{p}^\times$ we will distinguish two cases, namely whether or not $\G$ maps two-dimensional subspaces of $\V$ to two-dimensional subspaces $\V$. The former will be discussed in Section~\ref{semlinfct}, the latter will be considered together with the remaining case where $\Gamma \lneq \F{p}^\times$ in Section~\ref{remainCasFixZero}.

\subsection{The case where \texorpdfstring{$\G$}{G} \presprojlines}\label{semlinfct}

{In this section we will assume that $\Gamma$ is equal to $\F{p}^\times$ and  that $\G$ \emph{\presprojlines}.} Since for all $v\in \Vtimes$ the equivalence class $[v]_\sim$ is equal to the linear closure of $v$ without $\0$ and since the group $\G$  fixes $\0$, { the action of $\G$ on the $\sim$-equivalence classes is isomorphic to its action on
the set $\oneD{\V}$ of  one-dimensional subspaces  of $\V$.  It will follow that in this case the action of $\G$ on  $\oneD{\V}$ is the same as the action  of $\Aut(\V)$  thereon.}

\begin{definition}\label{planes_preserv_D}
Let $\W$ be a vector space. A bijective function   $g\colon \oneD{\W}\to \oneD{\W}$ \emph{\presprojlines} iff for all  $L_0,L_1,L_2\in \oneD{\W}$ we have
\begin{align}\label{plains_preserv}
     L_0\subseteq L_1 + L_2 \iff  L_0^g\subseteq L_1^g+L_2^g.
\end{align}
A group $\Hc\leq \Sym (\oneD{\W})$ \emph{\presprojlines}  iff all of its elements do. 
\end{definition}

{Let $\W$ be a vector space and let $g\colon \oneD{\W} \to \oneD{\W}$ be a bijective function which {\presprojlines}. Given a two-dimensional subspace $P\leq \W$, there are $L_1,L_2\in \oneD{\W}$ such that $P=L_1+L_2$, thus by (\ref{plains_preserv}) every one-dimensional subspace $L_0\subseteq P$ of $\W$ is mapped into $L_1^g+L_2^g$ by $g$ which, since $g$ is bijective on $\oneD{\W}$,  is a two-dimensional subspace. For any $L\in \oneD{\V}$ such that $L\subseteq L_1^g+L_2^g$ we have $L^{g^{-1}}\subseteq L_1 + L_2=P$. Therefore  we obtain}
\begin{align*}
\bigcup\{L^g: L\leq P \wedge L\in \oneD{\W}\}=L_1^g+L_2^g.
\end{align*}
In particular, $g$ maps the set of  one-dimensional subspaces of a given two-dimensional subspace of $\W$ (a ``projective line'') onto a set of the same form.

When we say that $\G$ \emph{\presprojlines} we indicate that the action of $\G$ on $\oneD{\V}$ \presprojlines. The following consequence of the Fundamental Theorem of Projective Geometry is a slight generalisation of Section 2.10 of~\cite[p.87ff]{Artin} from finite-dimensional to countably infinite-dimensional vector spaces, and follows from the statement there by a straightforward  induction.

\begin{theorem}\label{fun_thm_geom}
{Let $g \in \Sym(\oneD{\V})$ preserve projective lines. Then there exists a bijective linear function $\varphi \colon V \to V$  
such that the  action of $\varphi$ on  $\oneD{\V}$  is  $g$.  } 
\end{theorem}

Consequently, we have the following. 

\begin{proposition}\label{prop:res_col}
Let $\G$ be a closed supergroup of $\Aut (\V)$ fixing $\0$ which acts on  $\oneD{\V}$ and preserves projective lines. {Then  $\G$ acts on $\oneD{\V}$ like $\Aut (\V)$, i.e., the permutations on $\oneD{\V}$ obtained from the former action are precisely those of the latter.}\end{proposition}

\subsection{The remaining cases for \texorpdfstring{$\G\leq \Sym (V)_{\0}$}{G lowerequal Sym(V)}}\label{remainCasFixZero}
We  now consider the cases where we are not helped by the projective geometry being preserved, namely  
\begin{itemize}
    \item $\Gamma = \F{p}^\times$ and the action of $\G$ on $\oneD{\V}$ does not preserve projective lines, or
    \item $\Gamma \lneq \F{p}^\times$ (and hence $\G$ does not act on $\oneD{\V}$ at all).
\end{itemize}
{Our goal is to show that under these assumptions $\G$ acts on the set of those $\sim$-equivalence classes  which are not $\{\0\}$ as the full symmetric group.}

\begin{proposition}\label{prop:mainresult1} 
Assume that one of the following conditions holds.
\begin{itemize}
    \item $\Gamma = \F{p}^\times$ and the action of $\G$ on $\oneD{\V}$ does not preserve projective lines, or
    \item $\Gamma \lneq \F{p}^\times$.
\end{itemize}
Then $\G$ acts as the full symmetric group on $(V\setminus \{\0\})/_\sim$.
\end{proposition}

The proof is split into the following three steps:
\begin{enumerate}[label=(\roman*)]
    \item every tuple of vectors of $V\setminus \{\0\}$ can be mapped into a ``sufficiently small'' subspace of $\V$ by an element of $\G$;
    \item $\G$ acts transitively on arbitrary large tuples of  non-equivalent elements of $V\setminus \{\0\}$;
    \item $\G$ acts on $(V\setminus \{\0\})/_\sim$ as a closed permutation group.
\end{enumerate}
From (ii) and (iii)  it immediately  follows that $\G$ acts on $(V\setminus\{\0\})_\sim$ as the full symmetric group. 

The first step (i) will carry most of the weight.  The second step (ii) will follow quite easily from (i). \purple{Item} (iii) does not depend on (i) and (ii) and it is a straightforward consequence of the fact that all $\sim$-equivalence classes are finite. We start with a few basic observations for  stabilizers of $\G$.

\begin{lemma}\label{nagyorbit}
Let  $S$ be a finite subset of $V$. Then for all $v\in V$ the following are equivalent:
\begin{enumerate}
\item $G_S(v)$ is infinite.
\item $G_S(v)$ contains some vector $u\notin \linClos{S}$.
\item $G_S(v)\supseteq V\setminus \linClos{S}$.
\end{enumerate}
\end{lemma}

\begin{proof} Let $v\in V$ be given.

$(1) \Rightarrow (2)$:  Assume $G_S(v)$ to be infinite. Since $\langle S\rangle$ is finite $G_S(v)$ has to contain some element  outside of $\linClos{S}$. 

$(2) \Rightarrow (3)$: Already $\Aut(\V)_S$ acts transitively on the elements of $V\setminus \linClos{S}$.  Since there exists a vector in $G_S(v)$ which lies in $V\setminus \linClos{S}$, any element outside of $\langle S\rangle$ is also an element of  $G_S(v)$.

$(3) \Rightarrow (1)$: Since $V\setminus \linClos{S}$ is infinite, so is $G_S(v)$. \\
\end{proof}

\begin{lemma} \label{GSInftyIff} Let $S$ be a subset of $\Vtimes$. Then for all $v\in V$ and all $g\in G$ the following holds: 
\begin{enumerate}
    \item  $g^{-1}G_Sg=G_{S^g}$.
    \item  $G_S(v)^g=G_{S^g}(v^g)$.
    \item  $|G_S(v)|=\infty \iff|G_{S^g}(v^g)|=\infty$.
\end{enumerate}
\end{lemma}
\begin{proof}
Let $v\in V$ and $g\in G$ be given. If $h\in G_S$ then $g^{-1}hg$ fixes every element $s^g\in S^g$ since $s^{g(g^{-1}hg)}=s^{hg}=s^g$. For the same reason if  $f\in G_{S^g}$, then $g fg^{-1}$ fixes $S$ element-wise. This proves (1). 

For (2) we consider $G_S(v)^g$ which is equal to:
\begin{align*}
    \{v^h\colon h\in G_S\}^g &=\{v^{gfg^{-1}}\colon f\in G_{S^g}\}^g \\
   & =\{(v^g)^f\colon f\in G_{S^g}\}^{g^{-1}g}=G_{S^g}(v^g).
\end{align*}

Finally (3) follows immediately from (2).
\end{proof}

\begin{definition}\label{AKS}
Let  $S\subseteq \Vtimes$, and let $k\geq 1$. Then $A_k(S)$ is defined as the set of vectors $v\in V$ for which the set $\{v\}\cup S $ can be mapped into a $k$-dimensional subspace of $\V$ by an element of $\G$. \end{definition}

Note that 
for any set $S\subseteq V\setminus \{\0\}$ and any $k\geq 1$, we have  $A_k(S)\neq \emptyset$ if and only if $S$ can be mapped into a $k$-dimensional subspace of $\V$ by an element of $\G$; in that case,  $A_k(S)$ contains $\0$.

\begin{lemma}\label{propAkS}
Let $S$ be a subset of $\Vtimes$, and let $k\geq 1$. Then we have:
\begin{enumerate}
    \item For all $g \in G$ we have $A_k(S)=A_k(S^g)^{g^{-1}}$.
    \item For all $M \supseteq S$ we have $\ A_k(S) \supseteq A_k(M)$ \emph{(antitonicity)}.
    \item For all $v\in A_k(S)$ we have $\ G_S(v)\subseteq A_k(S)$. 
\end{enumerate}
\end{lemma}

\begin{proof}

(1): 
Let $g\in G$, and $v\in V$. Then clearly $\{v\} \cup S$ can be mapped into a $k$-dimensional subspace by an element of $\G$ iff $\{v^g\} \cup S^g$ can be mapped into a $k$-dimensional subspace by an element of $\G$. The equality $A_k(S)=A_k(S^g)^{g^{-1}}$ follows. 

(2): Let $M\supseteq S$ be given. For any $v\in V$ if the set $\{v\}\cup M$ can be mapped into some other given set by a function, then so can any subset of $\{v\}\cup M$ by the same function. In particular this is the case for $\{v\} \cup S$.

(3): Let $v$ be an element of  $A_k(S)$. Any element $w$ in $G_S(v)$ can be mapped to $v$ by some function of $\G_S$. Therefore, since $\{v\} \cup S$ can be mapped into a $k$-dimensional subspace by an element of $\G$, so can $\{w\}\cup S$.
\end{proof}

We now show that the set $A_k(S)$  has always  one of three shapes. 

\begin{lemma}\label{AkPoss}
Let $S$ be a subset of $\Vtimes$, and let $k\geq 1$. Then exactly one of the following holds:
\begin{enumerate}
 \item $A_k(S)=\emptyset. $
 \item There exists a $k$-dimensional subspace $W$ of $\V$ and a $g\in G$ such that $A_k(S)=W^g$.
 \item $A_k(S)=V$.
\end{enumerate}
\end{lemma}
\begin{proof}
 For $S=\emptyset$ the set $A_k(S)$ is equal to $V$, thus let $S\neq \emptyset$. If $S$ cannot be mapped into a $k$-dimensional subspace of $\V$ by any element of $\G$, then $A_k(S)=\emptyset$ which is the case (1). Otherwise  clearly $S\subseteq A_k(S)$.
 
 Let $g\in G$ be a function which maps $S$ into a $k$-dimensional subspace $W$ of $\V$. By Lemma~\ref{propAkS} (1) it suffices to prove the statement for  $A_k(S^g)$. The subspace $W$ is contained in $A_k(S^g)$. If $A_k(S^g)=W$, then case (2) follows. Otherwise there is an element $w \in A_k(S^g)\setminus W$. We  show that from this $A_k(S^g)=V$, i.e., case (3), follows.
 
 Since $\langle S^g \rangle \subseteq W$ the element $w$ cannot lie in $\langle S^g \rangle $, thus by Lemma~\ref{nagyorbit} the set $G_{ S^g}(w)$ contains $V\setminus W$. By applying Lemma~\ref{propAkS} (3) we obtain $V \setminus W \subseteq A_k(S^g)$ and (3) follows. 
\end{proof}

	Note that it is a fairly easy consequence of Proposition~\ref{prop:mainresult1} that case~(2) can hold for $S\subseteq \Vtimes$ and $k\geq 1$ only if the number of non-equivalent elements in $S$ is equal to the number of non-equivalent elements in a $k$-dimensional subspace of $\V$ minus one.
	

	We call cases~(1) and (3) in Lemma~\ref{AkPoss}  \emph{trivial}.  Some of the following properties are also true for the trivial cases but are so obviously and are of no significance for us.

\begin{lemma}\label{trans} Let $S\subseteq \Vtimes$ and let $k\geq 1$. Assume that $A_k(S)$ is nontrivial. Then:
\begin{enumerate}
    \item $S \subseteq \bigcup_{s\in S} [s]_\sim \subseteq A_k(S) \subseteq \linClos{S}$.
    \item $A_k(A_k(S))=A_k(S)$.
    \item For all $g\in G\colon$ if $S^g\subseteq A_k(S)$, then $A_k(S)=A_k(S)^g \ (=A_k(S^g))$.
\end{enumerate}
\end{lemma}
\begin{proof}
We prove (1). The first inclusion is trivial. Let $g\in G$ such that $g$ maps $S$ into a $k$-dimensional subspace $W$. Then $g\left(\bigcup_{s\in S}[s]_\sim \right)\subseteq W$.
{We show the last inclusion, $A_k(S) \subseteq \langle S\rangle$,  by  contradiction.} Suppose there exists an element $v \in A_k(S)$ which lies outside of $\langle S\rangle$. By Lemma~\ref{nagyorbit} we know that $|G_S(v)|=\infty$ and by Lemma~\ref{propAkS} and Lemma~\ref{AkPoss} we obtain that $A_k(S)$ equals $V$. This contradicts the non-triviality of $A_k(S)$, whence $A_k(S)\subseteq \linClos{S}$.

(2): The inclusion  $A_k(S)\subseteq A_k(A_k(S))$ follows from (1). The other inclusion follows because of antitonicity. 

(3): Let $g\in G$ be given. If $S^g\subseteq A_k(S)$ we obtain,   by antitonicity, that
\begin{align*}
A_k(S^g) \supseteq A_k(A_k(S))\overset{(2)}{=}A_k(S).
\end{align*}
	Lemma~\ref{propAkS} (1) implies that $A_k(S^g)=A_k(S)^g$ is finite and has the same size as $A_k(S)$. Therefore $A_k(S)=A_k(S^g)=A_k(S)^g$.
\end{proof}

The following corollary states that, if we assume that $S$ is a linearly independent subset of $V$, then the set $A_k(S)$ always contains specific linear combinations.

\begin{corollary}\label{trans2}
Let $S$ be a linearly independent subset of $V$, let $k,n\geq 1$, and let $(x_1,\ldots,x_n)$ and $(y_1,\ldots,y_n)$  be linearly independent tuples of  elements  of $A_k(S)$. Then for all $\lambda_1,\ldots,\lambda_n\in \mathbb{F}_{p}$ the following holds: 
\begin{align}
    \sum_{i=1}^n \lambda_i x_i \in A_k(S) \iff \sum_{i=1}^n \lambda_i y_i \in A_k(S).
\end{align}
\end{corollary}

\begin{proof}
In the cases of $A_k(S)=V$ and $A_k(S)=\emptyset$ this is trivially true. We assume otherwise. Since $\{x_1,\ldots,x_n\}\subseteq A_k(S) \subseteq \linClos{S}$ and $(x_1,\ldots,x_n)$ is linearly independent we have $\dim \linClos{S}\geq n$. Since $A_k(S)$ is nontrivial, it follows that $S\subseteq A_k(S)$ must be finite. Let $a_1,\dots,a_m$ be an enumeration of $S$. Then we can find $x_{n+1},\dots,x_m,y_{n+1},\dots,y_m\in S$ such that $\{x_1,\dots,x_m\}$ and $\{y_1,\dots,y_m\}$ are linearly independent. Then there exists $\varphi_x,\varphi_y \in \Aut(\V)$ such that $\varphi_x(a_i)=x_i$ and $\varphi_y(a_i)=y_i$ for all $i=1,\dots,m$. In particular $\varphi_x(S),\varphi_y(S)\subseteq A_k(S)$. Now applying Lemma~\ref{trans} (3) yields $A_k(S)=A_k(S)^{\varphi_x\varphi_y^{-1}}$. Since linear combinations are preserved by $\Aut (\V)$ the statement follows.
\end{proof}

{Corollary~\ref{trans2} will turn out to be an empty statement. If Proposition~\ref{prop:mainresult1} holds, as we already discussed, the only case in which $A_k(S)$ is non-trivial is if the number of equivalence classes of $\sim$ represented in $S$ is the same as the number of equivalence classes in $k$-dimensional subspaces of $\V$ minus one.  Any linear combination of vectors in $S$ with more than one coefficient unequal to zero would yield a vector non-equivalent to any element of  $S$. Therefore if $S$  is linearly independent  and $A_k(S)$ is non-trivial, it cannot contain any linear combination over $S$ with more than one coefficient unequal to zero. }

Clearly, for every element $v\in A_k(S)$ the entire equivalence class $[v]_\sim$ is contained in $A_k(S)$. Provided that $S$ is linearly independent we can show even more: unless $A_k(S)$ is already $V$, no other element of $\linClos{v}$ is contained in $A_k(S)$.

\begin{lemma}\label{nonEquivElemNot}
Let $S\subseteq V $ be linearly independent and let $k\geq 1$. Suppose that $A_k(S)$ is non-trivial. Then for all $v\in A_k(S)\setminus \{\0\}$:
\begin{align}\label{nurEuqivClass}
\forall \lambda \in \F{p}^\times: \lambda v \in A_k(S) \iff \lambda \in \Gamma.
\end{align}
\end{lemma}
\begin{proof} 
Let  $v\in A_k(S)\setminus \{\0\}$ be given. If $v\in \linClos{S}\setminus S$, then there exists  $\varphi\in \Aut (\V)$ and $w\in S$ such that $S^\varphi=S\setminus \{w\} \cup  \{v\}$. By Lemma~\ref{trans} (3) we obtain $A_k(S^\varphi)=A_k(S)$ . Therefore we can assume without loss of generality that $v\in S$.

The implication from right to left of (\ref{nurEuqivClass}) has already been shown in Lemma~\ref{trans} (1). For the other implication let us assume that $\lambda v \in A_k(S)$. Since $\lambda \not \in \Gamma$ there exists  $h\in G$ such that $\linClos{(\lambda v )^h}\neq \linClos{v^h}$. 
Since the dimension of $\V$ is infinite, there is a set $M\subseteq V$ of size $|S|-1$ such that both $M^h \cup \{v^h,(\lambda v)^h\}$ and ${M}\cup \{v\}$ are linearly independent.
There exists $\varphi\in \Aut(\V)$ such that $S^\varphi={M}\cup \{v\}$ and such that $v^\varphi=v$.  Then  $\linClos{(S\cup \{\lambda v\})^{\varphi h}}$ is a $|S|+1$-dimensional subspace of $\V$. Now $\lambda v\in A_k(S)$ implies $(S\cup \{\lambda v\})^{\varphi h}\subseteq A_k(S^{\varphi h})$ which is itself contained in $\linClos{S^{\varphi h}}$ by Lemma~\ref{trans} (1). The dimension of $\linClos{S^{\varphi h}}$ is at most $|S|$, thus this is a contradiction.

\end{proof}

We will need the following well-known fact which states that the affine subspaces of a vector space are precisely those subsets which are closed under the construction of  affine lines, if the underlying field is not $\F{2}$. 

\begin{proposition}\label{affinalter} 
Let $\W$ be 
a vector space over a field 
 different from $\F{2}$. Let  $A$ be a non-empty  subset of $W$ such that for all distinct $x,y \in A$ also $\Aff (x,y)\subseteq A$. Then $A$ is an affine subspace of $\V$.
\end{proposition}

If $S$ is a linearly independent subset of $V$ we use Proposition~\ref{affinalter} to show that if $A_k(S)$ contains an affine line not containing $\0$, then it contains certain affine spaces.

\begin{lemma}\label{isAffin}
Let $S\subseteq V$ be linearly independent. Let us assume that $A_k(S)$ contains an affine line not containing $\0$. Then for all $n\geq 1$ and any set $\{a_1,\ldots,a_n\}\subseteq A_k(S)$ of pairwise linearly independent elements, $A_k(S)$ contains the affine subspace $\Aff({a_1,\ldots,a_n})$ of $\V$.
\end{lemma}

\begin{proof}
Let a set of pairwise linearly independent elements  $\{a_1,\ldots,a_n\}$ $\subseteq A_k(S)$ be given. By  Proposition~\ref{affinalter} it suffices to show that for all distinct $1\leq i,j\leq n$ we have that the affine line $\Aff (a_i,a_j)$ is contained in $A_k(S)$.

{By assumption there exist distinct $v,w\in V$ such that the affine line $\Aff (v,w)$ is contained in $A_k(S)$ and $\0\not \in \Aff (v,w)$.}  Since $\Aff (v,w)$ does not contain $\0$ the set $\{v,w\}$ is linearly independent. Therefore, for all $\lambda \in \F{p}$ the  linear combination $\lambda v + (1-\lambda)w$ is an element of $A_k(S)$, and by applying Lemma~\ref{trans2} to the tuples $(v,w)$ and $(a_i,a_j)$  we obtain that every affine combination $\lambda a_i + (1-\lambda)a_j$ is also an element of $A_k(S)$. By Proposition~\ref{affinalter} we have Aff$(a_1,\ldots,a_n)\subseteq A_k(S)$.
\end{proof}

{The pairwise linear independence of the set $\{a_1,\ldots,a_n\}$ in Lemma~\ref{isAffin} is necessary. In particular, for a linearly independent set $S\subseteq V$ it does not hold in general that $A_k(S)$ is an affine subspace of $\V$. Since $A_k(S)$ contains $\0$, if $A_k(S)$ was an affine subspace of $\V$, it would also be a subspace of the vector space $\V$. For $\Gamma\neq \F{p}^\times$ this contradicts  Lemma~\ref{nonEquivElemNot}.}

	The following lemma gives a sufficient condition for $A_k(S)$ containing an affine line not containing $\0$.
	

\begin{lemma}\label{containsAffLine}
Let $S\subseteq V$ be linearly independent and let $k \geq 1$. Assume that $(a_1,\ldots,a_n)$ is linearly independent and such that $A_k(S)$ contains $a_1,\ldots,a_n$ as well as a vector  $v=\lambda_1a_1+\cdots+\lambda_n a_n$ with:
\begin{itemize}
\item not all coefficients $\lambda_1,\ldots,\lambda_n$ are the same,
\item at least three coefficients  are not equal to zero.
\end{itemize} Then $A_k(S)$ contains an affine line not containing $\0$.  
\end{lemma}

\begin{proof}
Without loss of generality we assume that $\lambda_1\neq \lambda_2$, and define $\lambda:=\lambda_1-\lambda_2\neq 0$. We apply Corollary~\ref{trans2} to the tuples $(a_1,a_2,a_3, \ldots ,a_n)$  and $(a_2,a_1,a_3, \ldots ,a_n)$ and we obtain that $v'\in A_k(S)$ where
\begin{align*}
v'& =\lambda_1 a_2 +\lambda_2a_1+\sum_{k=3}^n {\lambda_ka_k} \\
& =\sum_{k=1}^n{\lambda_ka_k}+\lambda(a_2-a_1)=v+\lambda(a_2-a_1).
\end{align*}

{Since $v$ as a linear combination of $a_1,\ldots,a_n$ has at least three coefficients not equal to zero, the set $\{v,a_1,a_2\}$ is linearly independent.} Thus the set $\{v',a_1,a_2\}$ is linearly independent as well. By reapplying Corollary~\ref{trans2} to the tuples $(v, a_2,a_1)$ and $(v',a_2,a_1)$ we obtain
\begin{align*}
 v+ \lambda (a_2-a_1)=v' &\in A_k(S) \\
\Longrightarrow  v'+\lambda( a_2-a_1)=:v'' & \in A_k(S).
\end{align*} 

On the other hand $v''$ is equal to $v+2\lambda (a_2-a_1)$. We continue in this fashion and eventually obtain that $v+\mu\lambda(a_2-a_1)\in A_k(S)$ for all $\mu\in \mathbb{F}_{p}$. Since $\lambda\neq 0$  we obtain \begin{align*}
L:=\{v+\mu (a_2-a_1): \mu\in  \mathbb{F}_{p}\} \subseteq A_k(S).
\end{align*} 
The set $\{v,a_2,a_1\}$ is linearly independent, hence the affine line $L$ does not contain $\0$.
\end{proof}

We remark that the proof of Lemma~\ref{containsAffLine} relies heavily on $\F{p}$ being a prime field, since any non-prime field  is not generated by $1$ and the addition. 
The next lemma states that it is sufficient to consider linearly independent tuples in Step~(i) from the beginning of Section~\ref{remainCasFixZero}.

\begin{lemma}\label{linInEnough}
Let  $n,k\geq 1$.  If some linearly independent $n$-tuple can be mapped  into a $k$-dimensional subspace of $\V$ by an element of $\G$, then every $n$-tuple can be mapped into a $k$-dimensional subspace of $\V$ by an element of $\G$.
\end{lemma}
\begin{proof}
We fix an arbitrary $k\geq 1$ and prove this statement via induction over $n$. The base case $n=1$ is obvious. For the induction step we assume the statement holds for $n$ and show it also holds for $n+1$.

{Assume there exists a linearly independent tuple $x$ of size $n+1$ which can be mapped into a $k$-dimensional subspace of $\V$. Every linearly independent tuple of size $n+1$ can be mapped to $x$ via an automorphism of $\V$. Therefore we only have to consider linearly dependent tuples.} 

Let  a tuple $t=(t_1,\ldots,t_{n+1})$ of  dimension smaller than $n+1$ be given. The tuple $t$ has at least one component which is a linear combination of all other components of $t$. Without loss of generality let $t_{n+1}$ be such a component.  By our induction hypothesis the initial segment $\tilde t = (t_1,\ldots,t_n)$ of $t$ can be mapped into a $k$-dimensional subspace $W$ of $\V$ by some $g\in G$. Moreover, $\linClos{\tilde t^g}$ is a subset of $W$, thus if $t_{n+1}^g\in \linClos{\tilde t^g }$, then  $t_{n+1}^g\in W$ and we are done.  Assume otherwise, i.e., $t_{n+1}^g \not \in \linClos{\tilde t^g}$. In that case, by Lemma~\ref{nagyorbit} we obtain $|G_{\{t_1,\ldots, t_n\}^g}(t_{n+1}^g)|=\infty$, whence $|G_{\{t_1,\ldots,t_n\} }(t_{n+1})|=\infty$ by Lemma~\ref{GSInftyIff} (3). 
Therefore there exists  $h\in G_{ \{t_1,\ldots,t_n\} }$ such that $t_{n+1}^h \not \in 
\linClos{\{t_1,\ldots,t_n\}}$, hence $\dim t^h = 1+ \dim t$.

If $t^h$ is linearly independent, we are done. Otherwise the size of $t^h$ is still $n+1$ and $\dim t^h <n+1$, thus we can repeat this process. Eventually we end with an element $g\in G$ such that $t^g$ is linearly independent. By assumption $t^g$ can be mapped into a $k$-dimensional subspace of $\V$ by an element of $\G$, whence so can $t$. \end{proof}

Let $S\subseteq V\setminus \{\0\}$ be of size $n\geq 1$. We want to map $S$  into a sufficiently small subspace of $\V$ by an element of $\G$.  {Since $\G$ preserves $\sim$-equivalence classes and $S$ might contain non-equivalent elements, such a  subspace of $\V$ has to contain at least $n|\Gamma|$ elements.} As it turns out for the sake of induction it will be convenient to have the subspace contain one additional equivalence class. This means $|\Gamma|$ additional elements. Since every subspace of $\V$ also contains $\0$, this sums up to at least $n|\Gamma|+|\Gamma|+1$ elements. For any $k\geq 0$, a $k$-dimensional subspace of $\V$ contains $p^k$ elements. Therefore the following inequalities have to hold:
\begin{align}
  &  (n+1)|\Gamma| +1 \leq p^{k} \iff \notag \\ \label{kandn}
    &   n \leq \frac{p^{k}-1}{|\Gamma|}-1.
\end{align}

Since $|\Gamma|$ divides $p-1$ it also divides $p^k-1$, thus the right-hand side of (\ref{kandn}) is always a natural number.

\begin{lemma}\label{step_i}
Assume that one of the following conditions holds.
\begin{enumerate}
    \item $\Gamma = \F{p}^\times$, and the action of $\G$ on $\oneD{\V}$ does not preserve projective lines,  or
    \item $\Gamma \lneq  \F{p}^\times$.
\end{enumerate}
Then for all $n\geq 3$ and all $k\geq 2$ satisfying (\ref{kandn}), every $n$-tuple of elements in $V\setminus \{\0\}$ can be mapped into a $k$-dimensional subspace of $\V$ by an element of $\G$.
\end{lemma}
\begin{proof}
{We proof this by induction over $n$}. For the proof assume $k$ to be the smallest natural number which is at least $2$ such that $(\ref{kandn})$ holds.  For the base case, $n=3$, we have to distinguish whether the assumption (1) or (2) holds.

\textbf{Case 1:} We assume that (1) holds. The inequality (\ref{kandn}) holds for $n=3$ and $k=2$. 
Thus we have to show the statement for $k=2$. There exists a function $g\in G$ whose action on $\oneD{\V}$ does not preserve projective lines. {Thus there exist distinct  one-dimensional subspaces $L_0\subseteq L_1+L_2$ of $\V$ such that $L_0^g \not \subseteq L_1^g +L_2^g$.} There are vectors $v_0,v_1,v_2\in L_1+L_2$ such that $\linClos{v_i}=L_i$,  for $i=0,1,2$. The tuple $(v_0^g,v_1^g,v_2^g)$ is linearly independent and mapped by $g^{-1}$ into the two-dimensional subspace $L_1+L_2$ of $\V$. By Lemma~\ref{linInEnough} we are done.

\textbf{Case 2:} We assume that (2) holds. Again $k=2$ since the right-hand side of (\ref{kandn}) only increases with smaller $\Gamma$. Moreover, because $\Gamma \lneq\F{p}^\times$, there are at least two elements $v,w\in V$ which lie in the same one-dimensional subspace of $\V$ and are split by some $g\in G$, i.e., $\langle v \rangle=\langle w \rangle$ and $\dim \langle \{v^g, w^g\} \rangle =2$. We now choose a vector $u\not \in \langle \{v^g, w^g\} \rangle$. Then $(v^g,w^g,u)$ is a linearly independent triple, which is mapped by $g^{-1}$ into the at most two-dimensional subspace $\langle \{u^{g^{-1}}, v,w\} \rangle$ of $\V$. Again by Lemma~\ref{linInEnough} we are done.

This concludes the base case. The right-hand side of (\ref{kandn}) is strictly increasing in $k$.  Therefore if in the induction step $n$ increases to $n+1$ and (\ref{kandn}) does not hold for $n+1$ and $k$ anymore, then it has to hold again for $n+1$ and $k+1$. In the base case $n=3$ and $k=2$, thus $n>k$ in every step of the induction. 

We assume we have already shown that every $n$-tuple can be mapped into a $k$-dimensional subspace of $\V$ by an element of $\G$. 

If  \begin{align*}
    n+1 > \frac{p^k-1}{|\Gamma|}-1, 
\end{align*}then we are done. This is because given an $(n+1)$-tuple $t$, {the $n$-tuple obtained by removing one component of $t$ can be mapped into a $k$-dimensional subspace of $\V$ by some $g\in G$ by our induction hypothesis.} Therefore the $(n+1)$-tuple $t$ is mapped by $g$ into a subspace  of $\V$ of dimension at most  $k+1$. 

Thus for the rest of the proof we  assume
\begin{align}\label{twoMore}
    n+1 \leq \frac{p^{k}-1}{|\Gamma|}-1 \iff n|\Gamma| + 2|\Gamma| \leq p^{k}-1.
\end{align}

By Lemma~\ref{linInEnough} it suffices to find one linearly independent $(n+1)$-tuple which is  mapped into a $k$-dimensional subspace  of $\V$ by some element of $\G$. Let $t=(t_1,\ldots,t_{n+1}) \in V^{n+1}$ be a linearly independent tuple, and let $\tilde t:=(t_1,\ldots, t_n)$. If $A_k(\{t_1,\ldots,t_n\})=V$, then we are done. 

Striving for a contradiction, suppose $A_k(\{t_1,\ldots,t_n\})\neq V$.   By the induction hypothesis $\tilde t$ can be mapped into a $k$-dimensional subspace $M$ of $\V$ by some $h\in G$, whence $A_k(\{t_1,\ldots,t_n\})\neq \emptyset$. By Lemma~\ref{AkPoss} there exists $g\in G$ and  a $k$-dimensional subspace $W$ of $\V$ such that $A_k(\{t_1,\ldots,t_n\})=W^g$. We want to apply Lemma~\ref{containsAffLine}. The tuple $t$ is linearly independent, hence so is $\tilde t$. 

\textbf{Claim:} $A_k(\{t_1,\ldots,t_n\})$ contains at least one element $v$ and a linearly independent set \linebreak $\{a_1,\ldots,a_n\}$ of size $n$ such that for some  $\lambda_1,\ldots, \lambda_n\in F$ we have $v=\lambda_1 a_1 + \cdots + \lambda_n a_n$ and
\begin{itemize}
\item not all coefficients $\lambda_1, \ldots, \lambda_n$ are the same,
\item at least three of them are non-zero.
\end{itemize}

{From the inequality (\ref{twoMore}) we obtain that $M$ contains two ``extra'' equivalence classes, i.e., when $\tilde t$ is mapped into the $k$-di\-men\-sion\-al subspace $M\leq \V$ by $h$, there exist two non-equivalent vectors $v,w\neq \0$ such that $[v^h]_\sim,[w^h]_\sim$ are contained in $M$ but  $t_1^h,\ldots,t_n^h$ are neither elements of $[v^h]_\sim$ nor of $[w^h]_\sim$.} In particular, $v,w$ are in $A_k(\{t_1,\ldots,t_n\})\subseteq \langle \{t_1,\ldots,t_n\}\rangle$ and non-equivalent to any component of $\tilde t$.

Thus, there exist coefficients $\lambda_1,\ldots,\lambda_n\in \F{p}$ and $\mu_1, \ldots,\mu_n \in  \F{p}$ such that  $v= \lambda_1 t_1 + \cdots + \lambda_n t_n$ and  $w=\mu_1 t_1 + \cdots + \mu_n t_n$. By Lemma~\ref{nonEquivElemNot} at least two coefficients of each of these linear combinations have to be non-zero. By the same lemma $\linClos{v} \neq \linClos{w}$, thus at least one set of coefficients cannot be all the same. Without loss of generality  assume that $\lambda_1,\ldots,\lambda_n$ are not all equal. If   more than three of these  coefficients are not equal to zero we are done. Otherwise  without loss of generality $\lambda_3=\cdots=\lambda_n=0$, hence 
\begin{align*}
    v=\lambda_1 t_1 +\lambda_2 t_2, \ \text{with} \ \lambda_1,\lambda_2 \neq 0.
\end{align*}

We apply Corollary~\ref{trans2} to the tuples $(t_1,t_2)$ and $(t_3,t_2)$. We obtain that $u :=\lambda_1 t_3 + \lambda_2 t_2$ is an element of $A_k(\{t_1,\ldots,t_n\})$ too. By combining the two equations we obtain
\begin{align*}
    v=\lambda_1 t_1 +( u-\lambda_1 t_3).
\end{align*}

{The set $\{t_1,u,t_3,t_4,\ldots, t_n\}$ is linearly independent and $v$ as a linear combination of these vectors has three coefficients which are  not zero.} The coefficients are not all equal since $\ch \F{p}\neq 2$. This shows our claim.

We are finally able to apply Lemma~\ref{containsAffLine}, which tells us that  $A_k(\{t_1,\ldots,t_n\})$ contains an affine line not containing $\0$. Moreover, by applying Lemma~\ref{isAffin} to the linearly independent set $\{t_1,\dots,t_n\}$ we obtain that $\Aff(\{t_1,\ldots,t_n\})$ is contained in $A_k(\{t_1,\ldots,t_n\})$.


We distinguish two cases: 

\textbf{Case 1: $\Gamma=\{1\}$:} In this case the $\sim$-equivalence class  of any $v\in V$ contains only $v$. Together with Lemma~\ref{nonEquivElemNot} this implies that for any $v$ in $A_k(\{t_1,\ldots,t_n\})$ the vector $\lambda v$ is an element of $A_k(\{t_1,\ldots,t_n\})$ iff $\lambda$ is either $1$ or $0$. Furthermore since $\linClos{\{t_1,\ldots,t_n\}} = \bigcup_{v\in \Aff ({\{t_1,\ldots,t_n\}})} \langle v \rangle$ and $A_k(\{t_1,\ldots,t_n\})\subseteq \linClos{\{t_1,\ldots,t_n\}}$, the set $A_k(\{t_1,\ldots,t_n\})$ is equal to $\Aff ({\{t_1,\ldots,t_n\}}) \cup \{\0\}$.  

{By Lemma~\ref{AkPoss} the set $A_k(\{t_1,\ldots,t_n\})$ contains $p^{k}$ elements. The set  $\Aff ({\{t_1,\ldots,t_n\}}) \cup \{\0\}$ contains $p^{(n-1)}+1$ elements.} This is a contradiction.

\textbf{Case 2: $\{1\}\lneq \Gamma $:} {Since Aff$(\{t_1,\ldots,t_n\})$ does not contain $\0$, it contains pairwise linearly independent elements.
By Lemma~\ref{nonEquivElemNot} for any element $v \in A_k(\{t_1,\ldots,t_n\})$ and any element $\lambda \in \Gamma$ also $\lambda v \in A_k(\{t_1,\ldots,t_n\})$. 
    We choose an arbitrary $\lambda\in \Gamma \setminus \{1\}$. Now $\Aff(\{t_1,\ldots,t_n\})\setminus\{t_1\}\cup \{\lambda t_1\}$  is pairwise linearly independent, thus its affine closure is contained in $A_k(\{t_1,\ldots,t_n\})$. We have
    \begin{align*}
        \Aff\left( \Aff(t_1,\ldots,t_n)\setminus \{t_1\}\right)=\Aff(t_1,\ldots,t_n).
    \end{align*} 
    Therefore $\Aff\left( \Aff(t_1,\ldots,t_n)\setminus \{t_1\}\cup \{\lambda t_1\}\right)$ is an affine subspace of $\V$ which contains $t_1$ and $\lambda t_1$ and in particular $\0$. Hence it is a subspace of $\V$ and we obtain $\langle \{t_1,\ldots,t_n\}\rangle \subseteq A_k(\{t_1,\ldots,t_n\}) \subseteq \linClos{\{t_1,\ldots,t_n\}}$.  
    
    The set $\linClos{\{t_1,\ldots,t_n\}}$ contains $p^{n}$ elements, while the cardinality of  $A_k(\{t_1,\ldots,t_n\})$ is $p^k$, which contradicts $n>k$.}
\end{proof}

This concludes  Step~(i) from the beginning of Section~\ref{remainCasFixZero}. The next step is to show that $\G$ acts transitively on arbitrary large tuples of  non-equivalent elements of $V\setminus \{\0\}$. To show this we restrict Lemma~\ref{step_i} to tuples containing   non-e\-quiv\-a\-lent elements such that the tuples are of specific but still strictly increasing sizes.

\begin{corollary}\label{tupOfSize} 
Under the assumptions of Lemma~\ref{step_i}, for all $k\geq 2$ every tuple $t$ which contains exactly $\frac{p^{k}-1}{|\Gamma|}-1$ non-equivalent non-zero elements of $\V$ can be mapped into a $k$-dimensional subspace $W$ of $\V$ by an element $g$ of $\G$. There is exactly one equivalence class $[v]_\sim$ in $W\setminus \{\0\}$ which is different from $[x^g]_{\sim}$ for all components $x$ of $t$.
\end{corollary}

\begin{lemma}\label{isTrans}
Assume that one of the following conditions holds.
\begin{enumerate}
    \item $\Gamma = \F{p}^\times$ and the action of $\G$ on $\oneD{\V}$ does not preserve projective lines, or
    \item $\Gamma \lneq  \F{p}^\times$.
\end{enumerate}
Then for every $n\geq 1$ the group $\G$  acts  $n$-transitively  on $(V\setminus \{\0\})/_\sim$. 
\end{lemma}
\begin{proof}
For every $k>2$ we define
\begin{align*}
   n_k:=\frac{p^{k}-1}{|\Gamma|}-1.
\end{align*} 

Since $(n_k)_{k>2}$ is a strictly increasing sequence it suffices to show that $\G$ acts $n_k$-transitively on $(V\setminus\{\0\})/_\sim$ for all $k>2$.

 Let $k>2$ and $t=(t_1,\ldots, t_{n_k})\in (V\setminus \{\0\})^{n_k}$ containing non-equivalent elements  be given. Since $\G$ acts transitively on linearly independent tuples of arbitrary size, it is enough to show that $t$ can be mapped to some linearly independent $n_k$-tuple by an element of $\G$.

We fix a linearly independent $n_k$-tuple $l=(l_1,\ldots,l_{n_k}) \in (V\setminus \{\0\})^{n_k}$. By Corollary~\ref{tupOfSize} both $t$ and $l$ can be mapped into a $k$-dimensional subspace of $\V$ by $g\in G$ and $h\in G$ respectively.  Since $\Aut (\V)\leq \G$ without loss of generality we may assume that both $t$ and $l$ are  mapped into the same $k$-dimensional subspace $W \leq \V$. 

Let $[v_t]_\sim \subseteq W\setminus \{\0\}$ be the unique equivalence class such that $t^g$ does not contain any element of $[v_t]_\sim$, and let $[v_l]_\sim \subseteq W\setminus \{\0\}$ be the unique equivalence class such that $l^h$ does not contain any element of $[v_l]_\sim$. {There exists an automorphism $\varphi$ of $\V$ which maps $v_t$ to $v_l$ and such that $\varphi(W)=W$.} The situation is depicted in Figure~\ref{fig:transitivity}.

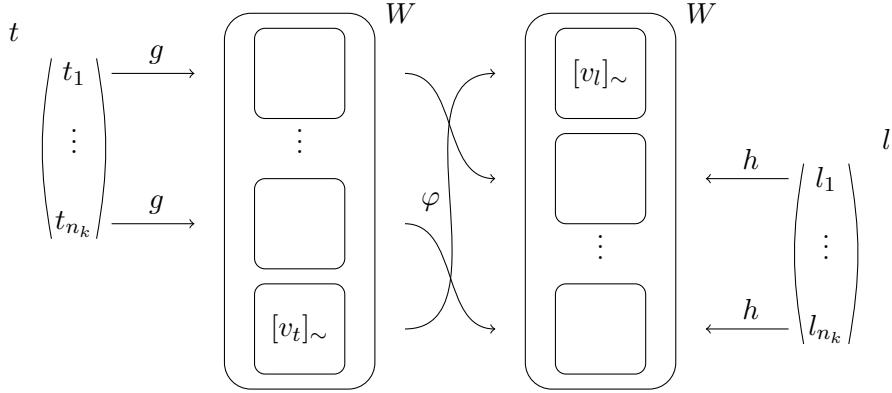
\begin{figure}
\begin{tikzpicture}
\node [below right] at (0,5) {$t$};
\node at (1,4.2) {$t_1$};
\node [above] at (1,3) {$\vdots$};
\node at (1,2.2) {$t_{n_k}$};
\draw (1.3,4.4) to [out=280,in=80] (1.3,2); 
\draw (0.7,4.4) to [out=260,in=100] (0.7,2);

\node [above] at (2.1,4.2) {$g$};
\node [above] at (2.1,2.2) {$g$};
\draw [->] (1.5,4.2) -- (2.6,4.2);
\draw [->] (1.5,2.2) -- (2.6,2.2);
\draw [rounded corners] (3.4,3.6) rectangle (4.6,4.8);
\draw [rounded corners] (3.4,1.6) rectangle (4.6,2.8);
\node [above] at (4,3) {$\vdots$};
\draw [rounded corners=4pt] (3.4,0.2) rectangle (4.6,1.4);
\node at (4,0.8) {$[v_t]_\sim$};
\draw [rounded corners=10pt] (3,0) rectangle (5,5);
\node [right] at (5,5) {$W$};

\node at (8,4.2) {$[v_l]_\sim$};
\draw [rounded corners] (7.4,3.6) rectangle (8.6,4.8);
\draw [rounded corners] (7.4,2.2) rectangle (8.6,3.4);
\node at (8,2) {$\vdots$};
\draw [rounded corners=4pt] (7.4,0.2) rectangle (8.6,1.4);
\draw [rounded corners=10pt] (7,0) rectangle (9,5);
\node [right] at (9,5) {$W$};

\node [below left] at (12,3.6) {$l$};
\node at (11,0.8) {$l_{n_k}$};
\node at (11,2) {$\vdots$};
\node at (11,2.8) {$l_1$};
\draw [->] (10.5,0.8) -- (9.4,0.8);
\draw [->] (10.5,2.8) -- (9.4,2.8);
\node [above] at (10,0.8) {$h$};
\node [above] at (10,2.8) {$h$};

\draw [->] (5.4,0.8) to [out=0,in=180] (6.6, 4.2);
\node [left] at (6,2.5) {$\varphi$};
\draw [->] (5.4,4.2) to [out=0,in=180] (6.6,2.8 );
\draw [->] (5.4,2.2 ) to [out=0,in=180] (6.6,0.8);

\draw (11.3,3) to [out=280,in=80] (11.3,0.6); 
\draw (10.7,3) to [out=260,in=100] (10.7,0.6);

\end{tikzpicture}
\caption{Transitivity of the action of $\G$ on the equivalence classes.}
\label{fig:transitivity}
\end{figure}

Thus, the tuples $t^{g\varphi h^{-1}}$ and $l$ are equal up to a permutation of their components. Since we only needed some linearly independent tuple to map $t$ to, this concludes the proof. 
\end{proof}

It remains to consider the  last step from beginning of Section~\ref{remainCasFixZero} which states that the action of $\G$  on $(V\setminus \{\0\})/_\sim$ is closed. A standard compactness argument using the finiteness of the equivalence classes of $\sim$  will show that the closedness of $\G$ when acting on $V$ is inherited by its action on $(V\setminus\{\0\})/_\sim$.

\begin{lemma}\label{isClosed}
Let $S$ be a countable set, let $\M\leq \Sym (S)$  be closed, and let $\simeq$ be an $\M$-invariant equivalence relation on $S$ such that all  its  equivalence classes are finite. Then  the action of $\M$ on the set $S/_\simeq$ of equivalence classes is closed.  
\end{lemma}

\begin{proof} 
Let $\{E_n\colon n\geq 0\}$ be an enumeration of the $\simeq$-equivalence classes  and let $(f_n)_{n\geq 0}$ be a sequence in $\M$ which converges in the action of $\M$ on $S/_\simeq$ to a function $f\in \Sym (S/_\simeq)$.

By restricting to a subsequence we may assume that for all $n\geq 0$ we have that for all $ k\geq n\colon f_k(E_n)=f(E_n)$. For every $n\geq 0$ the equivalence class $E_n$ is finite, thus there are infinitely many distinct  $i,j\geq n$ such that $f_i|_{E_n}=f_j|_{E_n}$. 
Again by restricting to subsequences we may assume that for all $n$ we have that $\forall i,j\geq n\colon  f_i|_{E_n}=f_j|_{E_n}$. Now the sequence $(f_n)_{n\geq 0}$ converges in the action of $\M$ on $S$ to a function $\tilde{f}\in \Sym (S)$. { Since $\M$ is closed we obtain $\tilde{f}\in M$. The action of $\tilde{f}$ on $S/_\simeq$ is equal to $f$, thus the action of $\M$ on $S/_\simeq$ is closed.}
\end{proof}

\begin{proof}[Proof of Proposition~\ref{prop:mainresult1}]
	Direct consequence of Lemmas~\ref{isTrans} and \ref{isClosed}.
\end{proof}

\subsection{Decomposition of $\G$ into two groups}\label{GisGtimesStuff}
Our results so far are summarized in Figure~\ref{fig:actionsonclasses}: we have obtained a complete understanding of the action of $\G$ on the equivalence classes of $\sim$.
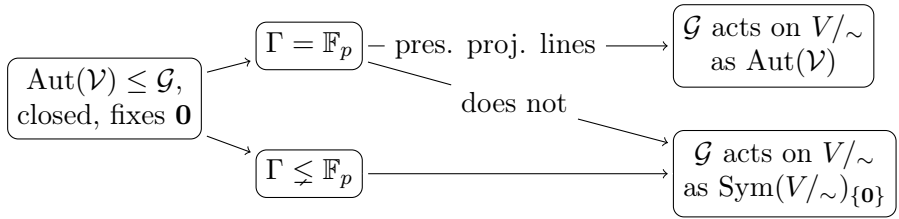
\begin{figure}
\begin{tikzpicture}[scale=1.07]
	\graph[grow right sep] {
	"$\Aut(\V)\leq \G$,\\closed, fixes $\0$" [rectangle, draw, rounded corners, align=center,xshift=-5mm]->[shorten >=2pt,shorten <=2pt]
	{
	n1/"$\Gamma=\F{p}$" [rectangle, draw, rounded corners, align=center,yshift=7mm] --[shorten <=2pt] "pres. proj. lines" [yshift=7mm] ->[shorten >=2pt]
	"$\G$ acts on $V/_\sim$\\ as $\Aut(\V)$" [rectangle, draw, rounded corners, align=center,yshift=7mm,xshift=7mm],
	n2/"$\Gamma\lneq\F{p}$" [rectangle, draw, rounded corners, align=center] --[color=white] /"\phantom{pres. proj. lines}" [color=white] ->[color=white] 
	"$\G$ acts on $V/_\sim$\\ as $\Sym(V/_\sim)_{\{\0\}}$" [rectangle, draw, rounded corners, align=center,xshift=7mm],
	n1 --[shorten <=2pt] does not[yshift=19.5mm,xshift=26mm] ->[shorten >=2pt]"$\G$ acts on $V/_\sim$\\ as $\Sym(V/_\sim)_{\{\0\}}$",
	n2 ->[shorten >=2pt,shorten <=2pt] "$\G$ acts on $V/_\sim$\\ as $\Sym(V/_\sim)_{\{\0\}}$",
	}
};
\end{tikzpicture}
\caption{Possible actions on the equivalence classes.}
\label{fig:actionsonclasses}
\end{figure}
We now  consider the elements of $\G$ as functions on $V$, and want to describe this action. In order to do this we decompose every element $g\in G$ into two parts, one that acts on $V/_\sim$ as the identity and one that may move the $\sim$-equivalence classes. We define 
\begin{align*}
    \Sstar (V):=\{g\in \Sym(V): g \text{ acts on } V/_\sim \text{ as the identity}\}.
    \end{align*}
Our goal of this section is to refine our results so far as indicated in Figure~\ref{fig:naturalactions}. 
\begin{figure}
\begin{center}
\begin{tikzpicture}[scale=1.07]
	\graph[grow right sep] {
		
		/"$\ldots$" [xshift=-3mm,yshift=7mm]->[shorten >=2pt]
		"$\G$ acts on $V/_\sim$\\ as $\Aut(\V)$" [rectangle, draw, rounded corners, align=center,yshift=7mm] ->[shorten <=2pt,shorten >=2pt]
		"$G=G^{\ast}\; \Aut(\V)$" [rectangle, draw, rounded corners, align=center,yshift=7mm,xshift=16mm], 
		/"$\ldots$" [xshift=-3mm]->[shorten >=2pt]
		"$\G$ acts on $V/_\sim$\\ as $\Sym(V/_\sim)_{\{\0\}}$" [rectangle, draw, rounded corners, align=center,]->[shorten <=2pt,shorten >=2pt]
		"$G=G^{\ast}\; \Sym((V\setminus\{\0\})/_\sim)^f$" [rectangle, draw, rounded corners, align=center,xshift=7mm], 
		/"$\ldots$" [xshift=-3mm,yshift=16mm]->[shorten >=2pt]
		"$\G$ acts on $V/_\sim$\\ as $\Sym(V/_\sim)_{\{\0\}}$" [rectangle, draw, rounded corners, align=center,],	
	};
\end{tikzpicture}
\end{center}\caption{Decomposing $\G$.}
\label{fig:naturalactions}
\end{figure}
Here, $G^\ast:=\Sstar(V)\cap G$, which clearly is  the domain of a closed subgroup $\Gast$ of $\G$. The set $\Sym((V\setminus\{\0\})/_\sim)^f$ will be a subgroup of $\Sym(V)$ which still acts like $\Sym(V/_\sim)_{\0}$ on $V/_\sim$ and fixes the $\sim$-equivalence classes setwise. We will define it later.

In the case when $\Gamma=\F{p}^\times$ and the action of $\G$ on $\oneD{\V}$ {\presprojlines} the decomposition turns out to be rather simple. We recall that in this case the equivalence classes of $\sim$ are exactly $\{\0\}$ and the one-dimensional subspaces of $\V$ without $\0$.

\begin{lemma}\label{GisGAut}
Assume that $\Gamma=\F{p}^\times$ and the action of $\G$ on $\oneD{\V}$ \presprojlines. Then
\begin{align*}
    G=G^\ast\Aut(\V).
\end{align*}
\end{lemma}
\begin{proof}
We only need to show the inclusion $G\subseteq G^\ast\Aut(\V) $. By Proposition~\ref{prop:res_col}    for all $g\in G$ there exists a function $\varphi \in \Aut(\V)\leq G$ such that the action $\varphi$ on $\oneD{\V}$ is the same as the one of $g$. Thus, $g\varphi^{-1}$ acts as the identity on $(V/\{\0\})/_\sim$ and is therefore an element of  $\Gast$.
\end{proof}

In order to talk about how an element $g\in G$ acts on elements within a given $\sim$-equivalence class we need to label the elements of the equivalence classes in a consistent way.


\begin{definition}
For any $\Lambda \leq \F{p}^\times$ a function $f\colon V \to \Lambda$ is a \emph{$\Lambda$-labelling} iff for all $v\in V$ and all $\lambda\in \Lambda$ we have $f(\lambda v)=\lambda f(v)$. For every $v\in V$  we call $f(v)$ the \emph{label of $v$}. 
\end{definition}

We recall that $\Gamma$ is the subgroup of $\F{p}^\times$ such that for all $v\in V$ the equivalence class $[v]_\sim$ is equal to $\Gamma v=\{\lambda v:\lambda\in \Gamma\}$. We can define a $\Gamma$-labelling $f\colon V \to \Gamma$ by choosing for each $\sim$-equivalence class  an arbitrary representative $v$  and setting $f(v):=1\in \Gamma$, as depicted in Figure~\ref{fig:label}. Since the equivalence class $[v]_\sim$ is equal to $\Gamma v$ this already determines  $f$ fully. 

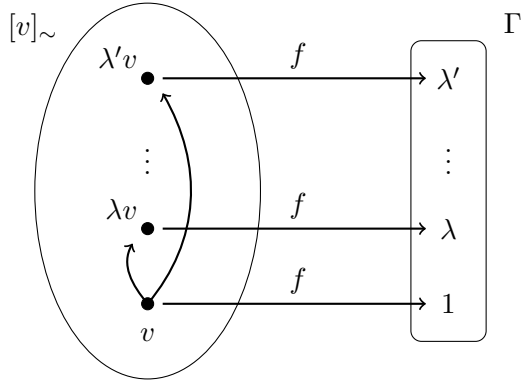
\begin{figure}\begin{center}
\begin{tikzpicture}
\node [below right, thick] at (0,5) {$[v]_\sim$};
\draw (2,2.5) ellipse [x radius=1.5, y radius=2.5];
\draw [fill] (2,1) circle  [radius=0.08];
\node [below] at (2,0.8) {$v$};
\draw [->, thick] (2,1) to [out=130,in=240] (1.8,1.8);
\draw [fill] (2,2) circle  [radius=0.08];
\node [above left] at (2,2) {$\lambda v$};
\draw [fill] (2,4) circle  [radius=0.08];
\node [above left] at (2,4) {$\lambda' v$};
\draw [->, thick] (2,1) to [out=50,in=300] (2.2,3.8);
\node at (2,3) {$\vdots$};

\draw [->, thick] (2.2,4) -- (5.7,4);
\draw [->, thick] (2.2,2) -- (5.7,2);
\draw [->, thick] (2.2,1) -- (5.7,1);
\node at (6,4) {$\lambda'$};
\node at (6,2) {$\lambda$};
\node at (6,1) {$1$};
\draw [rounded corners] (5.5,0.5) rectangle (6.5,4.5);
\node [below right] at (6.6,5) {$\Gamma$};
\node at (6,3) {$\vdots$};
\node [above] at (4,1) {$f$};
\node [above] at (4,2) {$f$};
\node [above] at (4,4) {$f$};
\end{tikzpicture}
\end{center}
\caption{A $\Gamma$-labelling.}
\label{fig:label}
\end{figure}

\begin{definition}
Let $g$ be an element of $\Sym ((V\setminus \{\0\})/_\sim)$ and let $f\colon V \to \Gamma$ be a $\Gamma$-labelling. Then the function $g^f\in \Sym (V)_{\0}$ is  defined by the following properties:
\begin{itemize}
    \item $g^f$ preserves $\sim$; 
    \item the action of $g^f$ on $(V\setminus \{\0\})/_\sim$ is equal to $g$; 
    \item for all $v\in V$ we have $f(v^{g^f})=f(v)$.
\end{itemize}
For a set $S\subseteq \Sym ((V\setminus \{\0\})/_\sim)$ we define $S^f:=\{g^f:g\in S\}$.
\end{definition}

Our goal is to show that, if $\Gamma\lneq \F{p}$ or $\Gamma=\F{p}$ and the action of $\G$ on $\oneD{\V}$ does not preserve projective lines, then for all $\Gamma$-labellings $f$ the set $\Sym ((V\setminus \{\0\})/_\sim)^f$ is contained in $G$. From this it will easily follow that $G=G^\ast\Sym ((V\setminus \{\0\})/_\sim)^f$.  { In order to do this we need to know how a given $g\in G$  acts on the ``inside of a $\sim$-equivalence class''.} We define the following.

\begin{definition}\label{sigma}
Let $f$ be a $\Gamma$-labelling, let $v\in V\setminus \{\0\}$ and let $g\in G$. A function $\sigma_f(g,v)\colon \Gamma \to \Gamma$ is defined by setting for any   $\lambda\in \Gamma$:
\begin{align}\label{sigmadef}
    \lambda \quad {\longmapsto} \ f\left( \left(\lambda\left( \frac{1}{f(v)}v\right)\right)^g\right).
\end{align}
\end{definition}

The function  performs the following steps.  
\begin{enumerate}
    \item The element $\frac{1}{f(v)} v$ is the element in $[v]_\sim$ which has label $1\in \Gamma$.
    \item The element $\lambda (\frac{1}{f(v)} v)=:x$ is the element in $[v]_\sim$ which has label  $\lambda$.
    \item $x^g$ is  an element in the $\sim$-equivalence class $g([v]_\sim)$ .
    \item Finally $f(x^g)$ is the label of $x^g$.
\end{enumerate}
In conclusion $\sigma_f(g,v)$ tells us how $g$ acts on the labels of the elements of $[v]_\sim$.

\begin{lemma}\label{sigma_prop}
Let $f$ be a $\Gamma$-labelling. Then for all $g\in G$ the following hold.
\begin{enumerate}
    \item For all $v\in V\setminus \{\0\}$ we have $\sigma_f(g,v)\in \Sym (\Gamma)$.
    \item For all $v,w\in V\setminus \{\0\}$ with $v\sim w$ we have $\sigma_f(g,v)=\sigma_f(g,w)$.
    \item The element $g$ is contained in $\Sym ((V\setminus \{\0\})/_\sim)^f$ iff for all $w\in V\colon \sigma_f(g,w)=\id_\Gamma$.
\end{enumerate}
\end{lemma}
\begin{proof}
All items follow directly from the definitions. 
\end{proof}

\begin{lemma} \label{sigma_circ}
Let $g,h$ be elements of $\G$ and let $f$ be a $\Gamma$-labelling. Then for all $v\in V\setminus \{\0\}$
\begin{align}
    \sigma_f(g h, v)=\sigma_f(g,v)\sigma_f(h,v^g).
\end{align}
\end{lemma}

\begin{proof} Let $v\in V$ be given.
By Lemma~\ref{sigma_prop} (2) we may assume without loss of generality that $f(v)=1$. There exist $u,w\in V$ such that  $[w]_\sim=[u^h]_\sim=[v^{gh}]_\sim$ and $f(u)=f(w)=1$. For all  $\lambda \in \Gamma$, there exist $\mu,\nu \in \Gamma$ such that $(\lambda v)^{gh}=(\mu u)^h=\nu w$. Then the situation is illustrated on the following diagram.

\begin{center}
\begin{tikzpicture}
\node at (0,3) {$\lambda v$};
\node at (4,3) {$(\lambda v)^g$};
\node at (8,3) {$(\lambda v)^{gh}$};
\draw [->,thick] (0.7,3) -- (3.3,3);
\node [above] at (2,3) {$g$};
\draw [->,thick] (4.7,3) -- (7.3,3);
\node [above] at (6,3) {$h$};
\draw [thick] (0.2,3.4) to [out=40,in=180] (4,4.3);
\draw [->,thick] (4,4.3) to [out=0,in=145] (7.8,3.4);
\node [above] at (4,4.3) {$gh$};

\draw [->,thick] (0,2.65)--(0,1.35);
\node [right] at (0,2) {$f$};
\draw [->,thick] (4,2.65)--(4,1.35);
\node [right] at (4,2) {$f$};
\draw [->,thick] (8,2.65)--(8,1.35);
\node [right] at (8,2) {$f$};

\node at (0,1) {$\lambda$};
\node at (4,1) {$\mu$};
\node at (8,1) {$\nu$};
\draw [->,thick] (0.7,1) -- (3.3,1);
\node [above] at (2,1) {$\sigma_f(g,v)$}; 
\draw [->,thick] (4.7,1) -- (7.3,1);
\node [above] at (6,1) {$\sigma_f(h,u)$}; 
\draw [thick] (0.2,0.6) to [out=-40,in=180] (4,-0.3);
\draw [->,thick] (4,-0.3) to [out=0,in=-145] (7.8,0.6);
\node [above] at (4,-0.3) {$\sigma_f(gh,v)$};
\end{tikzpicture}
\end{center}
Since $\sigma_f(h,v^g)=\sigma_f(h,u)$ this concludes the proof. 
\end{proof}
From  Lemma~\ref{sigma_circ} it follows immediately that for all $\Gamma$-labellings $f$, all $g\in G$ and all $v\in V$ the inverse of $\sigma_f(g,v)$ is $\sigma_f(g^{-1},v)$. Additionally Lemma~\ref{sigma_circ} shows that $\Sym ((V\setminus \{\0\})/_\sim)^f$ is a subgroup of $\Sym(V)$.

We are now able to find a function in $G\cap \Sym ((V\setminus \{\0\})/_\sim)^f$ which acts cyclically on an infinite linearly independent set of vectors.

\begin{lemma}\label{infiCycle}
Let $f$ be a $\Gamma$-labelling. Assume that $\G$ acts as the full symmetric group on the  equivalence classes $(V\setminus \{\0\})/_\sim$. { Then for all infinite and  linearly independent sets $\{v_i:i\in \Z\}\subseteq V$ such that $|V\setminus\linClos{\{v_i:i\in \Z\}}|=\infty$ } there exists $g\in G\cap \Sym ((V\setminus \{\0\})/_\sim)^f$ such that 
\begin{align*}
    \begin{array}{ccccccccc}
         \cdots & \overset{g}{\longrightarrow} & [v_{-1}]_\sim &\overset{g}{\longrightarrow} & [v_{0}]_\sim & \overset{g}{\longrightarrow} & [v_{1}]_\sim & \overset{g}{\longrightarrow} &\cdots
    \end{array}
\end{align*}
and $g$ acts as the identity on $(V/_\sim)\setminus\{[v_i]_\sim:i\in \Z\}$.
\end{lemma}

\begin{proof}
Let $\{v_i:i\in \Z\}\subseteq V$  be given. 
Without loss of generality  we may assume that for every $i\in \Z$ we have $f(v_i)=1$. Since $\G$ acts as the full symmetric group on $(V\setminus\{\0\})/_\sim$ there exists an element $s\in G$ such that for all $i\in \Z$ we have $s([v_i]_\sim)=[v_{i+1}]_\sim$ and such that $s$ acts as the identity on  $(V/_\sim) \setminus \{[v_i]_\sim:i\in \Z\}$. From this $s$ we now construct a function $g$ which acts like $s$ on the $\sim$-equivalence classes and which additionally satisfies  $\sigma_f(g,v)=\id_\Gamma$  for all $v\in V$.

Let $\varphi \in \Aut(\V)$ be such that $\varphi(v_i)=v_{i+1}$ for all $i\in \Z$.  Then  for all $i\in \Z$:  $\sigma_f(\varphi,v_i)=\id_\Gamma$. For all $n\geq 0$ we define $s_n:={\varphi^{-n}s\varphi^n}$. Since $s$ acts as the identity on $(V/_\sim)\setminus \{[v_i]:i\in \Z\}$ so does $s_n$. By Lemma~\ref{sigma_circ} we obtain for all $n\geq 0$ and all $i\in \Z$:
\begin{align*}
    \sigma_f(s_n,v_i)=\sigma_f(s,v_{i-n}).
\end{align*} 
Moreover, for all $n\geq 0$ an all $i \in \Z$ the function $s_n$ maps $[v_i]_\sim$ to $[v_{i+1}]_\sim$. 
We define $s':=s_0s_1\cdots s_{|\Gamma|!-1}$ and obtain for all $i\in \Z$:
\begin{align*}
    \sigma_f(s',v_i)&=\sigma_f(s_0,v_i)\sigma_f(s_1,v_{i+1})\cdots\sigma_f(s_{|\Gamma|!-1},v_{i+|\Gamma|!-1}) \\
    &= \sigma_f(s,v_i)^{|\Gamma|!}=\id_{\Gamma}.
\end{align*}
The last equality follows since $|\Sym(\Gamma)|=|\Gamma|!$. 

We found $s'$ such that for all $i\in \Z$ we have $\sigma_f(s',v_i)=\id_\Gamma$. {We now want $s''$ such that additionally $s''$ acts as the identity inside the equivalence classes $(V/_\sim)\setminus\{[v_i]:i\in \Z\}$}. For this we set $s'':=(s')^{|\Gamma|!}$. For all $i\in \Z$ we still have $\sigma_f(s'',v_i)=\id_{\Gamma}$. The function $s'$ acts as the identity on the $\sim$-equivalence classes which are not elements of  $\{[v_i]_\sim:i\in \Z\}$, therefore for every $u\in V$ which is not equivalent to an element in $\{v_i:i \in \Z\}$  we obtain
\begin{align*}
    \sigma_f(s'',u)=\sigma_f(s',u)^{|\Gamma|!}=\id_{\Gamma}.
\end{align*}
By Lemma~\ref{sigma_prop} (3) we obtain that $s''\in G\cap \Sym ((V\setminus \{0\})/_\sim)^f$. 

We set $k:=(|\Gamma|!)^2$. Then $s''$ acts on $\{v_i:i\in \Z\}$ in the following way:
\begin{align*}
    \begin{array}{ccccccccc}
         \cdots & \overset{s''}{\longrightarrow} & v_{-k} &\overset{s''}{\longrightarrow} & v_0 & \overset{s''}{\longrightarrow} & v_{k}& \overset{s''}{\longrightarrow} &\cdots, \\
         \cdots & \overset{s''}{\longrightarrow} & v_{1-k} &\overset{s''}{\longrightarrow} & v_{1} & \overset{s''}{\longrightarrow} & v_{1+k}& \overset{s''}{\longrightarrow} &\cdots, \\
         \ & \ & \ & \ & \vdots & \ & \  & \,     \\
         \cdots & \overset{s''}{\longrightarrow} & v_{(k-1)-k} &\overset{s''}{\longrightarrow} & v_{(k-1)} & \overset{s''}{\longrightarrow} & v_{(k-1)+k}& \overset{s''}{\longrightarrow} &\cdots
    \end{array}
\end{align*}

Let a finite set $F\subseteq V$ be given. Since $V\setminus \linClos{\{v_{i}:i\in \Z\}}$ is infinite there exists an automorphism $\varphi _F \in \Aut(\V)$ such that $\varphi_F$ restricted to $\linClos{ \{v_{jk}:j\in \Z\}}$ is the identity and \begin{align*}
    \varphi_F\left(F\setminus \linClos{ \{v_{jk}:j\in \Z\}}\right) \subseteq V\setminus \linClos{\{v_{i}:i\in \Z\}}.
\end{align*}
Since $s''$ acts as the identity on $V\setminus \linClos{\{v_{i}:i\in \Z\}}$ the composition ${\varphi_F s''\varphi^{-1}_F}$ acts as the identity on
\begin{align*}
    F \setminus \linClos{\{v_{jk}:j\in \Z\}}.
\end{align*}
Note that ${\varphi_Fs''\varphi^{-1}_F}$ preserves the labels of all elements in $F$. Since $F$ was arbitrary and since ${G\cap \Sym ((V\setminus \{\0\})/_\sim)^f}$ is closed  there exists a function $h$ therein such that 
\begin{align*}
    \begin{array}{ccccccccc}
         \cdots & \overset{h}{\longrightarrow} & v_{-k}\sim &\overset{h}{\longrightarrow} & v_{0}& \overset{h}{\longrightarrow} & v_{k} & \overset{h}{\longrightarrow} &\cdots,
    \end{array}
\end{align*}
and $h$ acts as the identity on the complement of  $\bigcup \{[v_{jk}]_\sim:j\in \Z\}$.

There exists $\varphi\in \Aut(\V)$ such that for all $i\in \Z$ we have $v_i^\varphi =v_{ik}$. Then the function $g:=\varphi h\varphi^{-1}$ proves our statement.
\end{proof}

With the help of Lemma~\ref{infiCycle} we are now able to  show  that any transposition of two equivalence classes can be achieved by a function in ${G\cap \Sym((V\setminus \{\0\})/_\sim)^f}$. From this, we easily obtain the following.

\begin{lemma}\label{GfisSymf}
Let $f$ be a $\Gamma$-labelling. Assume that $\G$ acts on $(V\setminus \{\0\})/_\sim$ as the full symmetric group. Then 
\begin{align*}
   \Sym((V\setminus \{\0\})/_\sim)^f\leq \G.
\end{align*}
\end{lemma}

\begin{proof}
First, let  two linearly independent vectors $v,w\in V$ such that $f(v)=f(w)$ be given. Let $\{v_i:i\in \Z\}$ be an infinite linearly independent subset of $V$ as in Lemma~\ref{infiCycle} and  such that $v_{-1}=v$ and $v_0=w$. By Lemma~\ref{infiCycle} there exists $g\in {G\cap \Sym((V\setminus \{\0\})/_\sim)^f}$ such that $g$ is the identity on $V\setminus \{[v_i]_\sim:i\in \Z\}$ and such that 
\begin{align*}
    \begin{array}{ccccccccc}
         \cdots & \overset{g}{\longrightarrow} & v_{-1} &\overset{g}{\longrightarrow} & v_0 & \overset{g}{\longrightarrow} & v_{1}& \overset{g}{\longrightarrow} &\cdots.
    \end{array}
\end{align*}

The set $\{v_i:i\in \Z\setminus\{0\}\}$ still satisfies the assumptions of Lemma~\ref{infiCycle}, thus there exists $h\in {G\cap \Sym((V\setminus \{\0\})/_\sim)^f}$ such that $h$ is the identity on $V\setminus \{[v_i]_\sim:i\in \Z\setminus\{0\}\}$ and such that 
\begin{align*}
    \begin{array}{ccccccccc}
         \cdots & \overset{h}{\longrightarrow} & v_{-1} &\overset{h}{\longrightarrow} & v_1 & \overset{h}{\longrightarrow} & v_{2}& \overset{h}{\longrightarrow} &\cdots.
    \end{array}
\end{align*}

For all $i\neq -1,0$ we have $v_i^{g h^{-1}}=v_{i+1}^{h^{-1}}=v_i$ and
\begin{align*}
    v_{-1}^{gh^{-1}}=v_{0}^{h^{-1}}=v_0=w \quad \text{ and } \quad v_{0}^{gh^{-1}}=v_{1}^{h^{-1}}=v_0=v.
\end{align*}

Therefore, since $g$ and $h$ act as the identity on $V\setminus \{[v_i]_\sim:i\in \Z\}$, we obtain that $gh^{-1}\in {G\cap \Sym((V\setminus \{\0\})/_\sim)^f}$ is the function which swaps exactly $[v]_\sim$ and $[w]_\sim$ while fixing everything else. We denote this function by $\pi_{[v]_\sim,[w]_\sim}$.

For two vectors $v\not\sim w$ such that $\linClos{v}=\linClos{w}$ we choose $u\in V\setminus \linClos{v}$. {By Lemma~\ref{infiCycle} and by what we showed $\pi_{[v]_\sim,[u]_\sim}$ and $\pi_{[w]_\sim,[u]_\sim}$ are both in  ${G\cap \Sym((V\setminus \{\0\})/_\sim)^f}$.} Since 
\begin{align*}
    \pi_{[v]_\sim,[u]_\sim}\pi_{[w]_\sim,[u]_\sim}\pi_{[v]_\sim,[u]_\sim} =\pi_{[v]_\sim,[w]_\sim}
\end{align*}
we obtain that every function which swaps exactly two $\sim$-equivalence classes and preserves labels is an element of ${G\cap \Sym((V\setminus \{\0\})/_\sim)^f}$. Considering that the smallest closed group containing all these elements is $\Sym((V\setminus \{\0\})/_\sim)^f$ we obtain $\Sym((V\setminus \{\0\})/_\sim)^f\subseteq G$.
\end{proof}

From here we show as in Lemma~\ref{GisGAut} that $\G$ decomposes into two groups.

\begin{corollary}\label{GisGSym}
Let $f$ be a $\Gamma$-labelling. Assume that $\G$ acts on $(V\setminus \{\0\})/_\sim$ as the full symmetric group. Then
\begin{align*}
    G=G^\ast\; \Sym((V\setminus \{\0\})/_\sim)^f.
\end{align*}
\end{corollary}
\begin{proof}
The inclusion $G\supseteq G^\ast\Sym((V\setminus \{\0\})/_\sim)^f$ holds by Lemma~\ref{GfisSymf}. For the other inclusion let an arbitrary $g\in G$ be given. Again by Lemma~\ref{GfisSymf} there exists  $g'\in \Sym((V\setminus \{\0\})/_\sim)^f$ which acts like $g$ on the $\sim$-equivalence classes, hence $gg'^{-1}\in G^\ast$.
\end{proof}

	It is clear from the definition that $G^\ast\cap\Sym((V\setminus \{\0\})/_\sim)^f=\{\id_V\}$ and that $G^\ast$ is a normal subgroup of $G$. Therefore, Corollary~\ref{GisGSym} tells us in fact that $G$ can be written as an (inner) semi-direct product $G^\ast\;  \Sym((V\setminus \{\0\})/_\sim)^f$.

\subsection{All groups fixing the zero vector.}

The goal of this section is to show that there are only finitely many reducts of $\V$ which define $\0$, i.e., there are only finitely many possibilities for the group $\G\leq \Sym(V)_{\0}$ being closed and containing $\Aut(\V)$.  In Section~\ref{GisGtimesStuff} we have seen in  Corollary~\ref{GisGSym} and Lemma~\ref{GisGAut} that $\G$ can always be decomposed into $\Gast$ and  $\Aut(\V)$ or $\Sym((V\setminus \{\0\})/_\sim)^f$, respectively. In this section we will show that $\Gast$ is already fully determined  by two subgroups $\Hc$ and $\Nr$ of $\Sym(\Gamma)$. Since $\Sym(\Gamma)$ is finite this will yield the desired result. 

{ For the rest of this section  we fix a $\Gamma$-labelling}    $f\colon V\to \Gamma$. We first define the sets $N$ and $H$ which will later turn out to be the domains of two permutation groups on $\Gamma$.
\begin{definition}\label{NandH}
The set $H$ is defined as the set of permutations $\sigma \in \Sym (\Gamma)$  such that  there exists a function $ g\in G^\ast $ and  a vector $v\in V$ such that $\sigma_f(g,v)=\sigma$. 
The set $N$ is defined as the set of permutations $\sigma \in \Sym (\Gamma)$ such that there exists such a function $g$ which additionally is the identity on $V\setminus [v]_\sim$. 
\end{definition}

 For any $\lambda\in \Gamma$ we also  write $\lambda$ for the permutation on $\Gamma$ which is induced by the multiplication with  $\lambda$.

\begin{lemma}\label{autgaut} 
{Let $\gamma$ be an  automorphism of $\V$, let $v\in V\setminus \{\0\}$  and let $g\in \Sstar (V)$.} Then
\begin{align*}
    \sigma_f(\gamma g \gamma^{-1},v)=\left(\frac{f(v^\gamma)}{f(v)}\right)\sigma_f(g,v^\gamma)\left(\frac{f(v)}{f(v^\gamma)}\right).
\end{align*}
\end{lemma}
\begin{proof}
By Lemma~\ref{sigma_circ} we obtain 
\begin{align*}
    \sigma_f(\gamma g\gamma^{-1},v)={\sigma_f(\gamma,v)}{\sigma_f(g,v^\gamma)}{\sigma_f(\gamma^{-1},v^{\gamma g})},
\end{align*}
and 
\begin{align*}
\id_\Gamma={\sigma_f(\gamma^{-1}\gamma,v^{\gamma g})}={\sigma_f(\gamma^{-1},v^{\gamma g})}{\sigma_f(\gamma, v^{\gamma g \gamma^{-1}})}.
\end{align*}

The function $g$ fixes the $\sim$-equivalence classes, hence $[v]_\sim=[v^{\gamma g \gamma^{-1}}]_\sim$. Combining those two equations we obtain
\begin{align*}
    \sigma_f(\gamma g\gamma^{-1},v)={\sigma_f(\gamma,v)}{\sigma_f(g,v^\gamma)}{\sigma_f(\gamma,v)^{-1}}.
\end{align*}

{Finally for all $\lambda\in \Gamma$ the function $\sigma_f(\gamma,v)$ maps $\lambda$ to
\begin{align*}
    f\left(\left(\lambda \frac{v}{f(v)}\right)^\gamma\right)=f\left(\lambda \frac{v^\gamma}{f(v)}\right)=\lambda \frac{f(v^{\gamma})}{f(v)}.
\end{align*}}
This concludes the proof.
\end{proof}

A special case of Lemma~\ref{autgaut} which we will use repeatedly  is the following: if $g\in \Gast$, and $\gamma\in \Aut(\V)$ maps a vector $v\in V\setminus\{\0\}$ to a vector of the same $f$-label, then 
\begin{align}\label{samelabelsigma}
    \sigma_f(\gamma g \gamma^{-1},v)=\sigma_f(g,v^\gamma).
\end{align}

	Next we show that any nonzero vector $v$ can be chosen as a witness for an element being in the group $H$ or $N$ as in Definition~\ref{NandH}.
	

\begin{lemma}\label{specificv}
Let $v$ be an element of  $V\setminus \{\0\}$. The following holds. 
\begin{itemize}
    \item For all $h\in H$ there exists  $g_h\in G^\ast$ such that $\sigma_f(g_h,v)=h$.
    \item For all $n\in N$ there exists $g_n\in G^\ast$ such that $\sigma_f(g_n,v)=n$ and such that $g_n$ is the identity on $V\setminus [v]_\sim$.
\end{itemize}
\end{lemma}
\begin{proof}
We  prove only the second item; the proof of the first item is identical.   Let $n\in N$ be given. By Definition~\ref{NandH} there exists $g\in G^\ast$ and $u\in V$ such that $\sigma_{f}(g,u)=n$ and for all non-zero vectors $w\in V\setminus[u]_\sim$ we have $\sigma_{f}(g,w)=\id_\Gamma$. Without loss of generality we assume $f(v)=f(u)=1$. Let  $\gamma$ be an automorphism  of $\V$ such that $v^\gamma=u$. We define $g':={\gamma g\gamma^{-1}}$. The function $g$ acts as the identity on $V/\sim$, hence $g'\in G^\ast$. { Moreover, since $g|_{V\setminus[u]_\sim}=\id_{V\setminus[u]_\sim}$ we obtain $g'|_{V\setminus[v]_\sim}=\id_{V\setminus[v]_\sim}$.
By Lemma~\ref{sigma_circ} and because $f(v)=f(v^\gamma)=1$ we have }
\[
    \sigma_{f}(g',v)=\sigma_{f}(\gamma, v)\sigma_{f}(g,u)\sigma_{f}(\gamma^{-1},v)=\id_\Gamma{n}\id_\Gamma=n. \qedhere
\]
\end{proof}

We can now show that $N$ and $H$ are the domains of two subgroups $\Nr$ and $\Hc$  of $\Sym (\Gamma)$ and that $\Nr$ is normal in $\Hc$.
\begin{lemma}\label{NHgroupsNormal}
$N$ and $H$ induce permutation groups $\Nr$ and $\Hc$ in $\Gamma$. \end{lemma}
\begin{proof}
Let $h_1,h_2$ be functions of $H$. Then by Lemma~\ref{specificv} there exist  a vector  $v\in V$  and  $g_1,g_2 \in G^\ast$ such that $\sigma_f(g_1,v)=h_1$ and $\sigma_{f}(g_2,v)=h_2$.  By Lemma~\ref{sigma_circ} since $[v^{g_1}]_\sim=[v]_\sim$ we obtain $\sigma_f(g_1g_2,v)=h_1h_2$, hence $h_1h_2 \in H$. Clearly the same holds for $N$,  {and so both $H$ and $N$ are domains of permutation groups $\Hc$ and $\Nr$ in $\Gamma$.}
\end{proof}

\begin{lemma}
$\Nr$ is normal in $\Hc$.
\end{lemma}
\begin{proof}
 Let $n\in N$ and $h\in H$ be given. Then by Lemma~\ref{specificv} there exists  a vector  $v\in V$  and $g_h,g_n\in G^\ast$  such that $\sigma_f(g_n,v)=n$ and $\sigma_f(g_h,v)=h$. Moreover, $g_h^{-1}g_ng_h\in G^\ast$ and we obtain
\begin{align*}
    \sigma_f(g_h^{-1}g_n g_h,v)= \sigma_f(g_h^{-1},v)\sigma_f(g_n,v)\sigma_f(g_h,v)= h^{-1}nh.
\end{align*}
Since $g_n$ fixes every element in $V\setminus[v]_\sim$ so does $g_h^{-1}g_ng_h$, hence $h^{-1}nh\in N$.
\end{proof}

We now want  to  reconstruct $\Gast$ from $\Nr$ and $\Hc$. If we manage to do so, then since $G=G^\ast\Aut(\V)$ or $G=G^\ast\Sym((V\setminus\{\0\})/\sim)^f$ we can conclude that there are only finitely many different closed groups fixing $\0$ between $\Aut(\V)$ and $\Sym(V)$.

\begin{definition}\label{defSG}
Let $\Nr'$ and $\Hc'$ be two permutation groups on $\Gamma$. We define $\SG{\Nr'}{\Hc'}$ as the set of all $g \in \Sstar(V)$ such that for all $v,w\in V\setminus\{\0\}$:
\begin{enumerate}
    \item $\sigma_f(g,v)\in H'$, and
    \item $\sigma_f(g,v)\sigma_f(g,w)^{-1}\in N'$.
\end{enumerate}
\end{definition}

Our goal is to show that $\SG{\Nr}{\Hc}$ is equal to $\Gast$.  By Lemma~\ref{GisGAut} and Corollary~\ref{GisGSym} there are two cases,  $G=G^\ast\Aut (\V)$ and  $G=G^\ast \Sym((V\setminus\{\0\})/_\sim)^f$. The second case turns out to be quite straightforward since $\G$ can map any $\sim$-equivalence class to any other $\sim$-equivalence class while still preserving the insides of the $\sim$-equivalence classes. We shall achieve the desired result for this case in Corollary~\ref{cor_SGH=GAST1}. After that, we shall treat the first case.

\begin{lemma}\label{ntonnnn}
Let  $S$ be a subset of $V\setminus \{\0\}$, and let $n\in N$. Then there exists $g\in G^\ast$ such that for all $v\in S$ we have $\sigma_f(g,v)=n$ and $g|_{V\setminus[S]_\sim}=\id_{V\setminus[S]_\sim}$.
\end{lemma}
\begin{proof}
Let $M\subseteq [S]_\sim$ such that $[M]_\sim=S$ and no two elements of $M$ are $\sim$-equivalent. By Lemma~\ref{specificv} for every $v \in M$ there exists $g_v \in G^\ast$ such that $\sigma_f(g_v,v)=n$ and $g_v|_{V\setminus[v]\sim}=\id_{V\setminus[v]\sim}$.
If $S$ is finite then the composition of all $g_v$ such that $v\in M$ in any order  proves the statement. 
Since $\Gast$ is closed, this holds for all infinite  $S\subseteq  V$ as well.
\end{proof}


\begin{lemma}\label{infidthenN}
Assume that $G=G^\ast\Sym((V\setminus\{\0\})/_\sim)^f$. Let  $S$ be an infinite subset of $ V$, and let $g\in G^\ast$.  If $g|_{[S]_\sim}=\id_{[S]_\sim}$, then for all $v\in V\setminus \{\0\}$ we have $\sigma_f(g,v)\in N$.
\end{lemma}
\begin{proof}
Let $v\in V\setminus \{\0\}$ be given. We need to  show the existence of  $g'\in G^\ast$ such that $\sigma_f(g',v)=\sigma_f(g,v)$ and $g'|_{V\setminus [v]_\sim}=\id_{V\setminus [v]_\sim}$. 
 Let $F$ be a finite subset of $V$. Since $\Sym((V\setminus\{\0\})/_\sim)^f\leq \G$ there exists $h_F\in G$ preserving labels such that $v^{h_F}=v$ and $h_F(F\setminus [v]_\sim) \subseteq [S]_\sim$. Clearly, the function $g_F':=h_Fgh_F^{-1}$ is an element of $\Gast$.
We obtain  $\sigma_f(g_F',v)=\sigma_f(g,v)$ and moreover $g_F'$ is the identity on $F\setminus [v]_\sim$.   Considering that $\Gast$ is closed we obtain that there exists an element $g'\in G^\ast$ with the desired properties.
\end{proof}

\begin{lemma}\label{infiniteThenEverythinginG}
Assume that $G=G^\ast\Sym((V\setminus\{\0\})/_\sim)^f$. Let $S$ be an infinite subset of $V\setminus \{\0\}$, and let $g\in G^\ast$.  If for all $v,w\in S$ we have $\sigma_f(g,v)=\sigma_f(g,w)=:h$, then there exists $g'\in G^\ast$ such that for all $u\in V\setminus \{\0\}$ we have $\sigma_f(g',u)=h$. 
\end{lemma}
\begin{proof}
For any finite $F\subseteq V\setminus \{\0\}$ there exists $h_F\in G$ preserving labels  such that  $h_F(F)\subseteq [S]_\sim$. Thus for all $v\in [F]_\sim$ we have $\sigma_f(h_Fgh^{-1}_F,v)=h$. Since $\G$ is closed   this proves the statement.
\end{proof}

\begin{lemma}\label{sigmasigmaSEC}
Assume that  $G=G^\ast\Sym((V\setminus\{\0\})/_\sim)^f$.  Let $g$ be an element of  $\Gast$. Then for all $v,w\in V$ the function $\sigma_f(g,v)\sigma_f(g,w)^{-1}$ is an element of $\Nr$.
\end{lemma}
\begin{proof}
{The group $\Sym (\Gamma)$ is finite,} therefore there is an infinite set $S\subseteq V$ such that for all $v,w\in S$ we have $\sigma_f(g,v)=\sigma_f(g,w)=:h\in \Sym (\Gamma)$. By Lemma~\ref{infiniteThenEverythinginG} there exists $g'\in G^\ast$ such that for every $u \in V$ the function  $\sigma_f(g',u)$ is equal to $h$. The composition $gg'^{-1}$ is the identity on $S$, thus by Lemma~\ref{infidthenN} for any given $v,w\in V$ both $\sigma_f(g,v)h^{-1}$ and $\sigma_f(g,w)h^{-1}$ are elements of $\Nr$.
By Lemma~\ref{NHgroupsNormal} we obtain that  $\sigma_f(g,v)\sigma_f(g,w)^{-1}=\left(\sigma_f(g,v)h\right)\left( \sigma_f(g,w)h\right)^{-1}$ is an element of $\Nr$.
\end{proof}

We thus have that assuming $G=G^\ast\Sym((V\setminus\{\0\})/_\sim)^f$, the group $\Gast$ can indeed be reconstructed from $\Nr$ and $\Hc$. 

\begin{corollary}\label{cor_SGH=GAST1}
Assume that $G=G^\ast\Sym((V\setminus\{\0\})/_\sim)^f$. Then $\SG{\Nr}{\Hc}=\Gast$.
\end{corollary}

{\begin{proof}
The inclusion $\SG{\Nr}{\Hc}\leq \Gast$ follows from Lemma~\ref{sigmasigmaSEC} and the definition of $\SG{\Nr}{\Hc}$. For the other inclusion let $g$ be an element of $\SG{\Nr}{\Hc}$. We show that for any finite subset $F$ of $V$ there exists $g'\in \Gast$ such that $g|_F=g'|_F$. Since $\Gast$ is closed this will be sufficient. 

Let $F$ be a finite subset of $V$. For some $n\geq1$ we fix a subset $M=\{v_1,\ldots,v_n\}$ of $F\setminus \{\0\}$ such that $[M]_\sim\supseteq F\setminus \{\0\}$ and for all distinct $i,j\leq n$ we have $[v_i]_\sim\neq[v_j]_\sim$. { By the Definition~\ref{defSG} we have $\sigma_f(g,v_1)\in H$,  thus there exists $\tilg\in G^\ast$ such that $\sigma_f(\tilg,v_1)=\sigma_f(g,v_1)$.} By Lemma~\ref{sigmasigmaSEC} we have for all $i\leq n$ that $\sigma_f(\tilg,v_1)\sigma_f(\tilg,v_i)^{-1}\in N$. Since $g\in \SG{\Nr}{\Hc}$ for all $i\leq n$ we have $\sigma_f(g,v_i)\sigma_f(g,v_1)^{-1}$. By $\sigma_f(g,v_1)=\sigma_f(\tilg,v_1)$ we obtain for all $i\leq n$ that $\sigma_f(g,v_i)\sigma_f(\tilg,v_i)^{-1}\in N$. {This implies for every $i\leq n$ the existence of $g_i\in G^\ast$ such that}
\begin{align*}
    \sigma_f(g_i,v_i)=\sigma_f(g,v_i)\sigma_f(\tilg,v_i)^{-1} \text{ and } g_i|_{V\setminus [v_i]_\sim}=\id_{V\setminus [v_i]_\sim}.
\end{align*}
We set $g':=(g_2g_3\cdots g_n )\tilg\in G^\ast$. Then for all $i\leq n$ we have
\begin{align*}
    \sigma_f(g',v_i)=\sigma_f(g,v_i).
\end{align*}
Thus $g|_F=g'|_F$ as desired. 
\end{proof}}

 We now consider the case  $G=G^\ast\Aut (\V)$, with the same goal of showing  $\SG{\Nr}{\Hc}=\Gast$. 

\begin{lemma}\label{ifallbuttwothenall}
There exists $N_G\geq 0$ such that for any  subspace $U\leq \V$ with $\dim U \geq N_G$, for any  $g\in G^\ast$ and for any $v\in U\setminus\{\0\}$ the following holds:
\begin{align*}
    \left(\forall w \in U\setminus \left([v]_\sim \cup \{\0\}\right): \sigma_f(g,w)\in   N\right) \implies \sigma_f(g,v)\in N.
\end{align*}
\end{lemma}

\begin{proof}
{We aim for a contradiction. Assume that the statement does not hold. Then there exist subspaces $(U_i)_{i\geq 0}$ of finite dimension such that $(\dim U_i)_{i\geq0}$ is strictly increasing, non-zero vectors $(v_i)_{i\geq 0}$ such that $v_i\in U_i$ and  functions $(g_i)_{i\geq 0}$ in $\Gast$ such that for all $i\geq 0$ we have
\begin{align*}
    \forall w \in U_i\setminus \left([v_i]_\sim\cup \{\0\}\right): \sigma_f(g_i,w)\in   N \text{ and } \sigma_f(g_i,v_i)\not \in N.
\end{align*}
We can assume that for all $i\geq 0$ we have $f(v_i)=1$. By the fact that $\Aut(\V)$ acts transitively on subspaces of a fixed dimension we may further assume that for all $i,j\geq 0$ we have $v_i=v_j=:v$ as well as 
\begin{align*}
    U_0\subsetneq U_1 \subsetneq \cdots  \text{ and } \bigcup_{i\geq 0} U_i = V.
\end{align*}

Repeated application of Lemma~\ref{ntonnnn} shows that for every $i\geq 0$ there exists $g_i'\in G^\ast$ such that
$\sigma_f(g_i',v)=\sigma_f(g_i,v)$ and $g_i'|_{U_i\setminus[v]_\sim}=\id|_{U_i\setminus[v]_\sim}$.
Since $\Sym(\Gamma)$ is finite, we are able to restrict ourselves to a subsequence  of $(g_i')_{i\geq 0}$ and $(U_i)_{i\geq 0}$ such that for all $i,j\geq 0$ we have $\sigma_f(g_i',v)=\sigma_f(g_j',v)=:\sigma$. Since $\Gast$ is closed,  there exists $g'\in G^\ast$ such that $\sigma_f(g',v)=\sigma$ and $g'|_{V\setminus [v]_\sim}=\id_{V\setminus [v]_\sim}$. This implies $\sigma\in N$ which is a contradiction and concludes our proof. }
\end{proof}

	For the rest of this section we fix a number $N_G\geq 2$ as in the conclusion of Lemma~\ref{ifallbuttwothenall}. The following observation is a direct consequence of the definition of $N_G$.
	
\begin{corollary}\label{ifallbuttwothenall_two}
For any subspace $U\leq \V$ with $\dim U>N_G$, for any  $g\in G^\ast$ and for any linearly independent vectors $v,w\in U$ the following holds:
\begin{align*}
    \left(\forall u\in U\setminus \left([v]_\sim \cup [w]_\sim \cup\{\0\} \right): \sigma_f(g,u)\in   N\right) \implies \sigma_f(g,v)\in N.
\end{align*}
\end{corollary}

\begin{proof}
	Apply the conclusion of Lemma~\ref{ifallbuttwothenall} to some $N_G$-dimensional subspaces $U_1,U_2$ of $U$ with $v\in U_1\setminus U_2$ and $w\in U_2\setminus U_1$.
\end{proof}

	We first show that if $G=G^\ast\Aut(\V)$, then  for all $g\in G^\ast$ and  all $v,w\in V\setminus\{\0\}$ the group $\Nr$ contains ${\sigma_f(g,v)\sigma_f(g,w)^{-1}}$. Striving for a contradiction we will assume that there exists $g\in G^\ast$ and $v,w\in V\setminus\{\0\}$ such that this is not the case. First we show that there exists some $g'\in G^\ast$ with the same property such that a certain set of commutators are contained in $N$. We will then use those commutators to derive a contradiction. 

\begin{lemma}\label{commutatorsinN}
Assume that $G={G^\ast\Aut(\V)}$. Let $U$ be a subspace of $\V$ of finite dimension $\dim U>N_G$.
 If there exist   a function $g\in G^\ast$ and $v,w \in U\setminus \{\0\}$ such that
\begin{align*}
    \sigma_f(g,v)\sigma_f(g,w)^{-1}\not \in N,
\end{align*}
then there exist $g'\in G^\ast$ and $v',w'\in U$ such that 
${{\sigma_f(g',v')}{\sigma_f(g',w')^{-1}}}\not \in N$ and for all $x,y\in U\setminus\{\0\}$ we have
\begin{align}\label{comminN}
    \sigma_f(g',x)\sigma_f(g',y)\sigma_f(g',x)^{-1}\sigma_f(g',y)^{-1} \in N.
\end{align}
\end{lemma}
\begin{proof}
For any $h\in G^\ast$ we define a  subset $W(h)$ of $U$ by
\begin{align}
    W(h):=\{v\in U:\sigma_f(h,v)\in N\}.
\end{align}
Clearly $[W(h)]_\sim=W(h)$. 
Let $h\in G^\ast$ be such that 
\begin{align}\label{propert1}
   \exists v,w\in U\text{ such that } \sigma_f(h,v)\sigma_f(h,w)^{-1}\not\in N
\end{align} and   such that $|W(h)|$ is maximal under all functions in $G^\ast$ satisfying (\ref{propert1}).

We first show that $W(h)$ contains at least one element. Let, as in the statement of the lemma, a function $g\in G^\ast$ and $v,w\in U$ be given such that ${\sigma_f(g,v)\sigma_f(g,w)^{-1}}\not \in N$. Since $\dim U\geq N_G+1\geq 3$ we can choose a vector $u\in U\setminus \linClos{v,w}$. Then there exists an automorphism $\varphi\in\Aut(\V)$ such that $v^\varphi=w$ and $\varphi$ fixes $u$. The function $\varphi g\varphi^{-1}g^{-1}$ is an element of $\Gast$ and \begin{align*}
    \sigma_f(\varphi g \varphi^{-1} g^{-1},v)=\sigma_f(g,v)\sigma_f(g,w)^{-1}\not \in N.
\end{align*}
therefore $\varphi g\varphi^{-1}g^{-1}$ satisfies (\ref{propert1}). On the other hand
\begin{align*}
    \sigma_f(\varphi g \varphi^{-1} g^{-1},u)=\sigma_f(g,u)\sigma_f(g,u)^{-1}=\id_\Gamma \in N,
\end{align*}
thus $|W(\varphi g \varphi^{-1} g^{-1})|\geq 1$. Because of maximality $|W(h)|\geq 1$.

If for any $x\in U$ we have $\sigma_f(h,x)\in N$, then by $\Nr\lhd \Hc$ we obtain for all $y\in V\setminus \{\0\}$
\begin{align*}
    {\underbrace{\sigma_f(h,x)}_{\in N}}{ \underbrace{\sigma_f(h,y){\overbrace{\sigma_f(h,x)^{-1}}^{\in N}}\sigma_f(h,y)^{-1}}_{\in N}}\in N.
\end{align*}
This shows that  for all $x,y\in U$, if 
\begin{align*}
    \sigma_f(h,x)\sigma_f(h,y)\sigma_f(h,x)^{-1}\sigma_f(h,y)^{-1} \not \in N,
\end{align*}
then neither $\sigma_f(g,x)$ nor $\sigma_f(g,y)$ is an element of $N$.

We now show that the chosen $h$ satisfies  (\ref{comminN}) for all $x,y \in U\setminus\{\0\}$. 
Striving for a contradiction we assume that there are vectors $x,y$ such that (\ref{comminN}) does not hold. Then $x,y\not\in W(h)$ by the argument above; moreover, clearly $x,y$ belong to distinct $\sim$-classes, which means that they are linearly independent since $G={G^\ast\Aut(\V)}$.  By Corollary~\ref{ifallbuttwothenall_two} it follows that we can find $z\in U\setminus ([x]_{\sim}\cup [y]_{\sim})$ such that $\sigma_f(g,z)\not\in N$, that is $z\not\in W(h)$. Without loss of generality we can assume that $f(x)=f(y)=f(z)$. We construct a function $h'$ satisfying (\ref{propert1}) such that $W(h)\subsetneq W(h')$ which contradicts the maximality of $W(h)$. 

There exists $\gamma \in \Aut (\V)$ such that $y^\gamma=x$ and $z^\gamma =w$ for some arbitrary fixed $w\in W(h)\neq \emptyset$. We define $g':=\gamma h\gamma^{-1}\in G^\ast$, and obtain
\begin{align*}
    \sigma_f(g',y)=\sigma_f(\gamma h \gamma^{-1},y)=\sigma_f(h,x) \text{ and } \sigma_f(g',z)=\sigma_f(\gamma h \gamma^{-1},z)=\sigma_f(h,w) \in N.
\end{align*}

Furthermore, we define $h':=hg'h^{-1}g'^{-1}\in G^\ast$ and  show  that $h'$  satisfies (\ref{propert1}). Since $\Nr\lhd \Hc$ we obtain for $z$ 
\begin{align} \label{forZ}
    \sigma_f(h',z)={\underbrace{\sigma_f(h,z){\sigma_f(h,w)}\sigma_f(h,z)^{-1}}_{\in N}}{\underbrace{\sigma_f(h,w)^{-1}}_{\in N}} \in N.
\end{align}

	Hence $z\in W(h')\setminus W(h)$. The elements $x,y$ are such that 
\begin{align*}
    \sigma_f(h',y)=\sigma_f(h,y)\sigma_f(h,x)\sigma_f(h,y)^{-1}\sigma_f(h,x)^{-1} \not \in N. 
\end{align*}
Since $\sigma_f(h',z)$ is an element of $\Nr$ we have $\sigma_f(h',y)\sigma_f(h',z)^{-1}\not \in N$, hence $h'$ satisfies (\ref{propert1}). 

For all $v\in W(h)$ by $\Nr\lhd \Hc$ we have
\begin{align*}
    \sigma(h',v)={\underbrace{\sigma_f(h,v)}_{\in N}}\underbrace{\sigma_f(g',v){\sigma_f(h,v)^{-1}}\sigma_f(g',v)^{-1}}_{\in N} \in N.
\end{align*}
We have seen already that $z\in W(h')\setminus W(h)$, hence $W(h)\subsetneq W(h')$. This contradicts the maximality of $W(h)$. Therefore for $h$  all $x,y\in U$ satisfy (\ref{comminN}).
\end{proof}

Since $\Nr\lhd \Hc$ we obtain under the assumptions of Lemma~\ref{commutatorsinN} that for any $k\geq 2$ and for all $x_1,\ldots,x_k,y_1,y_1\ldots,y_k\in U$ with $\{x_1,\dots,x_k\}=\{y_1,\dots,y_k\}$ there exists $n\in N$ such that
\begin{align}\label{ALLcomm}
   n\; \sigma_f(g',x_1)\cdots \sigma_f(g',x_k)  = \sigma_f(g',y_1)\cdots\sigma_f(g',y_k).
\end{align}

In the proof of Lemma~\ref{sigsiginN} below we will use a $\Gamma$-labelling which might differ from our fixed $\Gamma$-labelling. The following lemma shows how to switch between different labellings.

\begin{lemma}\label{labelsswap}
Let $f_1$ and $f_2$ be $\Gamma$-labellings. Then for all $g\in \Sstar (V)$, all $v\in V\setminus \{\0\}$ and $\lambda:=\frac{f_2(v)}{f_1(v)}$ we have
\begin{align}\label{labelswapform}
    \sigma_{f_2}(g,v)=\lambda^{-1}\sigma_{f_1}(g,v)\lambda.
\end{align}
\end{lemma}
\begin{proof}
We claim that for all $w\in [v]_\sim$ we have $f_2(w)=\lambda f_1(w)$.
This is true since there exists $\mu\in \Gamma^\times$ such that $\mu v=w$ and
\begin{align*}
    f_2(w)=\frac{f_2(w)}{f_1(w)} f_1(w)=\frac{f_2(\mu v)}{f_1(\mu v)} f_1(w) =\frac{\mu f_2(v)}{\mu f_1(v)} f_1(w)=\lambda f_1(w).
\end{align*}

Let $w\in [v]_\sim$ be such that $f_2(w)=1$, then $f_1(\lambda w)=1$. Now for any $\mu\in \Gamma$ we have $\sigma_{f_2}(g,v)(\mu)=f_2((\mu v)^g)$ and 
\begin{align*}
    \sigma_{f_1}(g,v)(\mu)&=f_1((\mu(\lambda w))^g)=\lambda^{-1}f_2(((\lambda \mu)w)^g)\\
    &=\lambda^{-1}\sigma_{f_2}(g,v)(\lambda \mu) = (\lambda\sigma_{f_2}(g,v)\lambda^{-1})(\mu). \qedhere
\end{align*}
\end{proof}

For the proof of the next lemma we need for every $n$-dimensional subspace $U$ of $\V$ an element in $\Aut(U)$ of order $p^n-1$, a so-called \emph{Singer cycle}. {Such an element always exists since we may identify $U$  with  $\F{p^n}$, viewed as a vector space; since the multiplicative group  $\F{p^n}^\times$ of the field $\F{p^n}$
 is cyclic, it is generated by a single element, which then has to have order $p^n-1$.} Since the multiplication by this element is an   automorphism of the vector space $\F{p^n}$, this shows that there exists an automorphism of $U$ of order $p^n-1$. 

We will also refer to Euler's Phi function $\varphi\colon \N \to \N$ which assigns to every $n\in \N$ the number of positive integers $k\leq n $ which are coprime to $n$.


\begin{lemma}\label{sigsiginN}
Assume that $G=G^\ast\Aut(\V)$. Let  $g$ be an element of $\Gast$. Then for all $v,w\in V\setminus\{\0\}$ the function $\sigma_f(g,v)\sigma_f(g,w)^{-1}$ is an element of $\Nr$.
\end{lemma}
\begin{proof}
We strive for a contradiction. Assume that there exists $g\in G^\ast$ and  $\ut,\uh \in V\setminus\{\0\}$ such that $\sigma_f(g,\ut)\sigma_f(g,\uh)^{-1}\not \in N$. Without loss of generality let  $\ut$ and $\uh$ be such that their $f$-labels are $1$. Let $n$ be some natural number greater that $N_G$ and which is divisible by $\varphi((p-1)!)$.

{ Let $U$ be an $n+1$-dimensional subspace of $\V$ which contains $\{\ut,\uh\}$, and let $A$ be an $n$-dimensional affine subspace of $U$ which  also contains $\{\ut,\uh\}$ but does not contain $\{\0\}$. We set $U':=A-\ut=A-\uh$ which is a subspace of   $U$.} Since $A$ does not contain $\0$, it follows that the set $A$ is linearly independent. Thus, there exists a $\Gamma$-labelling which is constant on $A$. By Lemma~\ref{labelsswap} since $\lambda\in H$ and $\Nr\lhd \Hc$ for any $\lambda\in \F{p}$ we may assume that $f$ already coincides with this $\Gamma$-labelling. Moreover, by Lemma~\ref{commutatorsinN} we may also assume that 
for all $x,y\in U$ we have
\begin{align}\label{commagain}
    \sigma_f(g,x)\sigma_f(g,y)\sigma_f(g,x)^{-1}\sigma_f(g,y)^{-1} \in N. 
\end{align}

As we argued above we know that there exists $\gamma\in \Aut(U')$ of order $p^n-1$, that is for any $u\in U'\setminus \{\0\}$ we have
\begin{align}\label{Uprimeisgammaa}
    \{u^{\gamma^i}: 0\leq i\leq p^n-2\}=U'\setminus \{\0\}.
\end{align}

Let $\gt$ and $\gh$ in $\Aut(\V)$ be extensions of $\gamma$ such that $\ut^{\gt}=\ut$ and $\uh^{\gh}=\uh$. The set $U'=\gt(U')$ is equal to ${\gt}(A-\ut)={\gt}(A)-\ut$, thus ${\gt}(A)=A$. By similar reasoning  ${\gh}(A)=A$. We define for all $i\leq p^n-2$ the functions
\begin{align*}
    \gtt_i:=\gt^ig\gt^{-i} \text{ and }  \ghh_i:=\gh^ig\gh^{-i}.
\end{align*}

For any $a\in A$ and all $i\leq p^n-2$ we obtain $\sigma_f(\gtt_i,a)=\sigma_f(g,a^{\gt^i})$ and $\sigma_f(\ghh_i,a)=\sigma_f(g,a^{\gh^i})$. We further define
\begin{align*}
    \htt:=\gtt\gtt_1\gtt_2\cdots\gtt_{p^n-2} \text{ and }
    \hhh:=\ghh\ghh_1\ghh_2\cdots\ghh_{p^n-2}.
\end{align*}
 Because of (\ref{Uprimeisgammaa}) we obtain for all $a\in A\setminus \{\ut\}$: 
\begin{align*}
    \{a^{\gt^i}: 0\leq i\leq p^n-2\}=A\setminus \{\ut\},
\end{align*}
and likewise for all $a\in A\setminus \{\uh\}$ we have $\{a^{\gh^i}: 0\leq i\leq p^n-2\}=A\setminus \{\uh\}$. 

We enumerate $A$ such that $A=\{u_i: 0\leq i \leq p^n-1\}$ and $\ut=u_{0}$ as well as $\uh=u_{1}$.  
Recall that $n$ was chosen to be a multiple of $\varphi((p-1)!)$. By Euler's Theorem we know that $(p-1)!$ divides $p^{\varphi((p-1)!)}-1$, \purple{and thus it also} divides $p^{n}-1$. The function $\sigma_f(\htt,\ut)$ is equal to $\sigma_f(g,\ut)^{p^n-1}=\id_{\F{p}^\times}$.

For all $a\in A\setminus \{\ut\}$ we arrive at
\begin{align*}
    \sigma_f(\htt,a)=\sigma_f(g,u_{k_1})\cdots\sigma_f(g,u_{k_{p^n-1}})
\end{align*}
such that $\{u_{k_i}\colon 1\leq i \leq p^n-1\}=A\setminus\{\ut\}$. {By Lemma~\ref{commutatorsinN} and by what we remarked in (\ref{ALLcomm}) we obtain that for every $a\in A\setminus \{\ut\}$ there exists $n_a\in N$ such that}
\begin{align*}
    n_a\sigma_f(\htt,a) =\left(\sigma_f(g,u_{0})\cdots\sigma_f(g,u_{{p^n-1}})\right)\sigma_f(g,\ut)^{-1}=:\stt.
\end{align*}

Likewise we define $\shh:={\left(\sigma_f(g,u_{0})\cdots\sigma_f(g,u_{{p^n-1}})\right)}{\sigma_f(g,\uh)^{-1}}$. With this $\stt^{-1}\shh=\sigma_f(g,\uh)\sigma_f(g,\ut)^{-1}$ which is by assumption not an element of $\Nr$.
This implies that either  $\stt$ or $\shh$ is not an element of $\Nr$. Without loss of generality we can assume that $\stt\not \in N$. Let $\psi\in \Aut(\V)$ be a function which fixes every element in $U'$ and maps $\ut$ into $A\setminus \{\ut\}$. We define $h:=\htt\psi\htt^{-1}\psi^{-1}\in G^\ast$. For all $v\in U'\setminus[\ut]_\sim$ we obtain
\begin{align*}
    \sigma_f(h,v)&=\sigma_f(\htt,v)\left(\sigma_f(\psi \htt \psi^{-1},v)\right)^{-1}\\
    &=\sigma_f(\htt,v)\sigma_f(\htt,v)^{-1}=\id_{\Gamma}\in N.
\end{align*}

{For any $v\in U\setminus (U'\cup [\{{\ut},\ut^{\psi^{-1}}\}]_\sim)$ there exists $a\in A\setminus \{\ut,\ut^{\psi^{-1}}\}$ such that $\linClos{a}=\linClos{v}$.} We obtain
\begin{align*}
    \sigma_f(h,v)&=\sigma_f(\htt,a)\left(\sigma_f(\psi \htt \psi^{-1},a)\right)^{-1}\\
    &=\sigma_f(\htt,a)\sigma_f(\htt,a^\psi)^{-1}=(n_a^{-1}\stt)(n_{a^\psi}^{-1}\stt)^{-1} \\
    &=n_a^{-1}n_{a^\psi}\in N.
\end{align*}

Therefore for all elements $v\in U\setminus [\{\ut,\ut^{\psi^{-1}}\}]_\sim$ we have $\sigma_f(h, v)\in N$, hence by Corollary~\ref{ifallbuttwothenall_two} also $\sigma_f(h,\ut)\in N$. On the other hand
\begin{align*}
    \sigma_f(h,\ut)&=\sigma_f(\htt,\ut)\left(\sigma_f(\psi \htt \psi^{-1},\ut)\right)^{-1}\\
    & = \sigma_f(\htt,\ut^\psi)^{-1}=(n_{\ut^{\psi}}^{-1}\stt)^{-1}.
\end{align*}
The second-to-last equality follows since $\sigma_f(\htt,\ut)=\id_\Gamma$. This contradicts $\stt\not \in N$ and therefore concludes the proof.
\end{proof}

{
\begin{corollary}
Assume that $G=G^\ast\Aut(\V)$. Then $\SG{\Nr}{\Hc}=\Gast$.
\end{corollary}
\begin{proof}
The proof is identical to the one of Corollary~\ref{cor_SGH=GAST1} except that it refers to Lemma~\ref{sigsiginN} instead of Lemma~\ref{sigmasigmaSEC}.
\end{proof}}

{We have shown  that  for every closed group $\G$ such that $\Aut(\V)\leq \G \leq \Sym(V)_{\0}$ 
the group $\Gast$ is equal to $\SG{\Nr}{\Hc}$ from Definition~\ref{defSG},} thus $\Gast$ is uniquely determined by its corresponding groups $\Nr \lhd \Hc\leq \Sym (\Gamma)$ from Definition~\ref{NandH}. Since there are only finitely many possibilities for $\Nr$ and $\Hc$ as subgroups of $\Sym(\Gamma)$ we have shown the following. 

\begin{proposition}\label{fin_red_fix_0}
Assume that a closed permutation group $\G$ on $V$ containing  $\Aut(\V)$  fixes the zero vector $\0$. Let $f$ be an arbitrary $\Gamma$-labelling.  Then $G=G^\ast\Sym((V\setminus\{\0\})/_\sim)^f$ or $G=G^\ast\Aut(\V)$, and in either case $G^\ast=\SG{\Nr}{\Hc}$ for some $\Nr \lhd \Hc\leq \Sym (\Gamma)$. In particular, the number of such groups is finite.
\end{proposition}

\section{The closed supergroups of \texorpdfstring{$\Aut (\V)$}{Aut(V)} moving  \texorpdfstring{$\0$}{0}}\label{not fix 0}\label{sect:move0}

	In Section~\ref{sect:fix0} we have shown that there exist only finitely many reducts of $\V$ which do first-order define the zero vector $\0$. The goal for this section is to show that there are only two reducts of $\V$ which do not first-order define $\0$, \purple{namely the infinite dimensional affine space over $\F{p}$ and the pure set}.	For the rest of the section we fix a closed permutation group $G$ acting on $V$ such that 
\begin{align*}
    \Aut (\V) \leq \G\leq \Sym (V) \text{ and } \exists g\in G: g(\0)\neq \0.
\end{align*}
The group $\Aut (\V)$ acts transitively  on $V\setminus \{\0\}$. Since  $\0$ is not fixed by $\G$ we immediately obtain that $\G$ acts transitively  on $V$.

If we consider the subgroup $\Gr_{\0}$ of $\G$ we may apply all results of the previous sections. With this in mind we consider the  equivalence relation $\sim$ given by $\Gr_{\0}$ as before together with $\Gamma \leq \F{p}^\times$ as defined before.  We will show that in our setting $\Gamma$ can only be one of the two trivial subgroups of $\F{p^\times}$.

\begin{lemma}\label{lem:mainres_G_is_sym}
Assume that  $\Gamma=\{1\}$. Then $\G=\Sym (V)$.
\end{lemma}
\begin{proof}
By Proposition~\ref{prop:mainresult1} the group $\Gr_{\0}$ acts as $\Sym ({(V/_\sim)})_{\0}$ on the $\sim$-equivalence classes of non-zero vectors. Since these classes contain only one element and there exists $g\in G$ such that $g(\0)\neq \0$, the claim follows. \end{proof}

We henceforth assume $\Gamma\neq \{1\}$.


\begin{definition}
Let $S$ be a set and  let $\Hc$ be a permutation group acting on a set $S$. For all $X \subseteq S$ the \emph{algebraic closure $\acl(X)$ of $X$} is defined to be the union of all finite orbits of $H_{X}$. 
\end{definition}

As for the linear closure, for finitely many elements  $v_1,\ldots,v_n\in V$ we are going to  write $\acl (v_1,\ldots,v_n)$ instead of $\acl(\{v_1,\ldots,v_n\})$. 
%
It is clear from the definition that $\acl$ is a \emph{closure operator, that is for all subsets $X,Y \subseteq S$ we have \begin{itemize}
\item $X\subseteq \acl(X)$
\item if $X\subseteq Y$ then $\acl(X)\subseteq \acl(Y)$
\item $\acl(\acl(X))=\acl(X)$.
\end{itemize}
}

In our case the set $S$ is the set of vectors $V$ and the group $\Hc$ is $\G$. Since $\Aut (\V)\leq \G$ it follows that for all $u,v,w\in V$ such that $u\in V\setminus \linClos{v,w}$ the orbit $\G_{v,w}(u)$ already contains $V\setminus \linClos{v,w}$, in particular it is infinite. 
Thus, we have 
\begin{align}\label{agcl_linclos}
    \{v,w\}\subseteq\acl (v,w)\subseteq\linClos{v,w}.
\end{align}
Our goal is show that in fact $\acl(v,w)=\Aff(v,w)$ for any two vectors $v,w\in V$. 

\begin{lemma}\label{agcl_size}
Let $v, w$ be distinct vectors in $V$ and let $g\in G$.  Then 
\begin{enumerate}
    \item $\acl (v,w)^g=\acl(v^g,w^g)$, and
    \item $|\acl (v,w)|=|\Gamma|+1$.
\end{enumerate}
\end{lemma}
\begin{proof}

Clearly, for all $u\in V$ we have
\begin{align*}
 |G_{v,w}(u)|=\infty \iff |G_{v^g,w^g}(u^g)|=\infty.
\end{align*}
This shows (1).

	It follows easily from the definition of $\sim$ that $\acl(\0,u)=[u]_{\sim}\cup \{\0\}$ for any $u\in V\setminus \{\0\}$. Since $\G$ is 2-transitive the general case follows from item (1).
\end{proof}

\begin{lemma}\label{GammaIsOneOfTwo}
	$\Gamma=\F{p}^\times$.
\end{lemma}

\begin{proof}
	Let $v$ be an arbitrary nonzero vector in $V$. Since $\Gamma\neq \{1\}$ we can choose $\hat{v}\in [v]_\sim$ different from $v$. Clearly, any orbit of $G_{v,\hat{v}}$ outside $[v]_\sim\cup \{\0\}$ is infinite. Therefore $\acl(v,\hat{v})\subseteq [v]_\sim \cup \{0\}$. On the other hand by Lemma~\ref{agcl_size} we know that $\acl(v,\hat{v})=|\Gamma|+1$. Therefore $\acl(v,\hat{v})=[v]_\sim \cup \{\0\}$.
	
	Let us assume for contradiction that $\Gamma\neq \F{p}^\times$. Then by 
	Proposition~\ref{prop:mainresult1} we know that the group $\G_{\0}$ acts on the set of $\sim$-equivalence classes as the full symmetric group. Let $g\in G$ such that $u:=\0^g\neq \0$. Let $w\in V$ such that $[w]_\sim$ is disjoint from $([u]_\sim\cup \{\0\})^{g^{-1}}$, and let $\hat{w}\in [w]_\sim \setminus \{w\}$. Then we have $u,w^g,\hat{w}^g\neq \0$ and $[u]_\sim \neq [w^g]_{\sim},[\hat{w}^g]_{\sim}$. By our argument in the previous paragraph we know that $\0\in \acl(w,\hat{w})$, and thus by Lemma~\ref{agcl_size} (1) we have $u=\0^g\in \acl(w^g,\hat{w}^g)$. We claim that $w^g$ and $\hat{w}^g$ are also not $\sim$-equivalent. Indeed, if $w^g\sim \hat{w}^g$ then by applying the result of the previous paragraph with $v=w^g$ and $\hat{v}=\hat{w}^g$ we would obtain $\acl(w^g,\hat{w}^g)=[w^g]_{\sim}\cup \{\0\}$. This would imply in particular that $u\in [w^g]_{\sim}$, a contradiction. We have obtained that $u,w^g,\hat{w}^g$ are in pairwise different $\sim$-classes. Then these vectors can be mapped to three linearly independent vectors by some element of $G_{\0}$. This is however impossible since $u\in \acl(w^g,\hat{w}^g)$.
\end{proof}

	The following is an immediate consequence of Lemma~\ref{GammaIsOneOfTwo} and item (2) of Lemma~\ref{agcl_size}.

\begin{corollary}\label{agcl_size_p}
	Let $v, w$ be distinct vectors in $V$. Then $|\acl (v,w)|=p$.
\end{corollary}


	It also follows immediately from Lemma~\ref{GammaIsOneOfTwo} that for all $v\in V\setminus \{\0\}$ we have $\acl(\0,v)=[v]_\sim \cup \{\0\}=\linClos{v}$.
For arbitrary vectors $v,w\in V$ the structure of $\acl (v,w)$ is a little more complex to determine. For all $u\in V$, if there is a function in $\G_{v,w}$ which maps $u$ to a vector outside of $\linClos{v,w}$ then $u$ cannot be in $\acl (v,w)$. On the other hand, if the orbit $G_{v,w}(u)$ is fully contained in $\linClos{v,w}$, then it is finite and $u$ is an element of the algebraic closure. Therefore 
\begin{align}\label{orbits_infinite_iff}
    \acl (v,w)=\{u\in V\colon G_{v,w}(u)\subseteq \linClos{v,w}\}.
\end{align}

{
\begin{lemma}\label{agcl_is_agcl}
Let $v,w$ be distinct vectors in $V$. For all  distinct elements $x,y$ of $  \acl (v,w)$, we have
\begin{align*}
    \acl (x,y)=\acl (v,w).
\end{align*}
\end{lemma}}
\begin{proof}
	Since $x,y\in \acl(v,w)$ we have $\acl(x,y)\subseteq \acl(v,w)$. By Lemma~\ref{agcl_size_p} we know that $|\acl(x,y)|=|\acl(v,w)|=p$. Therefore $\acl (x,y)=\acl (v,w)$.
\end{proof}

	Together with item~(1) of Lemma~\ref{agcl_size} we obtain the following. 

\begin{corollary}\label{invariant_unter_g_falls_drinnen}
Let $v,w\in V$ distinct and let $g\in G$ be a function such that $v^g, w^g\in \acl (v,w)$. Then $\acl (v,w)^g=\acl (v,w)$.
\end{corollary}

\begin{lemma}\label{acl_nicht_equiv}
{Let  $v,w$ be linearly independent vectors in $V$. Then every pair $(x,y)$ of distinct  elements of  $\acl (v,w)$ is linearly independent. }
\end{lemma}
\begin{proof}
If $x,y$ were distinct linearly dependent elements in $\acl(v,w)$, then by Lemma~\ref{agcl_is_agcl} we would obtain $v,w\in\acl(v,w)=\acl(x,y)=\linClos{x}$ which is impossible.
\end{proof}

\begin{lemma}\label{acl_linkomb}
Let $v,w\in V$ be two vectors which are linearly independent. For all distinct $x, y\in \acl (v,w)$ and all $\lambda, \mu\in \F{p}$ we have: if $\lambda v + \mu w\in \acl (v,w)$, then $\lambda x + \mu y \in \acl (v,w)$.
\end{lemma}

\begin{proof}
Let $x, y\in\acl (v,w)$  be distinct and $\lambda,\mu\in \F{p}$ such that $\lambda v+ \mu w\in \acl (v,w)$ be given. The set $\{x,y\}$ is linearly independent.Therefore, there exists $\varphi\in \Aut (\V)$  such that $v^\varphi=x$ and $w^\varphi=y$. By Corollary~\ref{invariant_unter_g_falls_drinnen} the set $\acl (v,w)^\varphi$ is $\acl (v,w)$. Hence, $(\lambda v + \mu w)^\varphi=\lambda x +\mu y$ is an element of $\acl (v,w)$.
\end{proof}

\begin{lemma}\label{acl_is_aff}
Let $v,w$ be two vectors in $V$.  Then the algebraic closure  $\acl(v,w)$ is the affine line $\Aff(v,w)$.
\end{lemma}
\begin{proof}
 If $\{v,w\}$ is linearly dependent, then since $|\acl (v,w)|=p$ and $\acl (v,w)\subseteq \linClos{v,w}=\linClos{v}$, we have $\acl (v,w)=\linClos{v}=\Aff (v,w)$.

We assume that $\{v,w\}$ is linearly independent, and we define a set $I$ as the set of all pairs $(\lambda,\mu)\in \F{p}^2$ such that there exist distinct $x,y\in \acl(v,w)$ for which we have: $\lambda x +\mu y \in \acl (v,w)$.

 Because of Lemma~\ref{acl_linkomb} for all $(\lambda, \mu)\in \F{p}^2$ the existence of distinct  elements $x,y\in \acl(v,w)$ such that $\lambda x+\mu y\in \acl(v,w)$ is equivalent to all distinct elements $x,y\in \acl(v,w)$ fulfilling $\lambda x + \mu y \in \acl(v,w)$.

If $(\lambda, \mu)\in I$, then $\frac{1}{\lambda}(\lambda v + \mu w)-\frac{\mu}{\lambda}w=v$ and since all three $(\lambda v + \mu w),v,w$ are elements of $\acl (v,w)$ we obtain
\begin{align}\label{acl_is_aff1}
   \forall (\lambda, \mu)\in I, \lambda\neq 0 \implies \left(\frac{1}{\lambda},-\frac{\mu}{\lambda}\right)\in I.
\end{align}

Furthermore for all $(\lambda, \mu)\in I$ such that $\lambda\neq 0$, the element  $\lambda ( \frac{1}{\lambda}v- \frac{\mu}{\lambda}w)+\mu v$ is contained in $\acl(v,w)$ and equal to $ v- \mu w + \mu v= (1+\mu) v -\mu w$. Thus 
\begin{align}\label{acl_is_aff2}
    \forall (\lambda, \mu)\in I, \lambda\neq 0 \implies (1+\mu,-\mu)\in I.
\end{align}
If $1+\mu\neq 0$ then by applying (\ref{acl_is_aff2}) again we obtain $(1-\mu,\mu)\in I$, and thus also $(\mu,1-\mu)\in I$.

\textbf{Claim:} {For all $\mu\in \F{p}$ if $(0,\mu)\in I$, then $\mu=1$.}

This is clear since $(0,1)\in I$ and by Lemma~\ref{acl_nicht_equiv} two linearly dependent elements cannot lie in $\acl (v,w)$.

\textbf{Claim:} For all $\lambda \in \F{p}$ if $(1,\lambda)\in I$, then $\lambda=0$.

We assume otherwise, i.e., $\lambda\neq 0$. By (\ref{acl_is_aff1}) we obtain 
\begin{align*}
    \left(\frac{1}{\lambda},-\frac{1}{\lambda}\right), \ (1,-\lambda)\in I.
\end{align*}

Therefore 
\begin{align*}
    \acl(v,w)\ \ni \ \frac{1}{\lambda}\underbrace{(v+\lambda w)}_{\in \ \acl(v,w)}-\frac{1}{\lambda}\underbrace{(v-\lambda w)}_{\in  \ \acl(v,w)}=2v .
\end{align*}
This is again impossible by Lemma~\ref{acl_nicht_equiv} since $2v\neq v$.

\textbf{Claim:} For all $(\lambda,\mu)\in I$ if $\lambda \neq 0$, then $\lambda=1-\mu$.

If $\lambda\neq 0$, then we have seen that $(\mu,1-\mu)\in I$, that is  $\mu v+ (1-\mu) w\in \acl(v,w)$, and
\begin{align*}
     1\cdot (\mu v+ (1-\mu) w) + (\lambda - (1-\mu)) w =\mu v+\lambda w\in \acl (v,w) 
\end{align*}
Thus, $(1, \lambda-(1-\mu))\in I$, and then our second claim implies $\lambda=1-\mu$.

Our three claims show that $\acl(v,w)\subseteq \Aff(v,w)$. By Corollary~\ref{agcl_size_p} we have $|\acl(v,w)|=p$ which is also equal to $|\Aff(v,w)|$. Therefore $\acl(v,w)= \Aff(v,w)$.
\end{proof}

Lemma~\ref{agcl_size} (1) and Lemma~\ref{acl_is_aff} immediately show the following. 

\begin{corollary}\label{Cor_pres_aff_lines}
 The elements of $\G$ preserve affine lines, i.e., they map affine lines to affine lines. 
\end{corollary}

It thus remains to apply the following formulation of   Fundamental Theorem of Affine Geometry. See for example~\cite[p.52]{Berger} for a proof in the finite-dimensional case; alternatively, it can easily be derived from Theorem~\ref{fun_thm_geom}.

\begin{theorem}\label{thm:fundthmaffine} Let $g\in\Sym(V)$ preserve affine lines. Then it preserves affine combinations, and is the composite of a translation and an automorphism of $\V$.
\end{theorem}

Theorem~~\ref{thm:fundthmaffine} shows that $\G\leq \AGL(\V)$. The other inclusion is easy.

\begin{lemma}
{ $\G$ is equal to the group of all  affine bijective mappings  from $V$ to $V$.}
\end{lemma}

\begin{proof}
By Corollary~\ref{Cor_pres_aff_lines} and Theorem~\ref{thm:fundthmaffine} we have $\G\leq \AGL(\V)$. Since $\G$ does not fix $\0$, it contains element which is the composition of a vector space automorphism and a translation by a non-zero vector. Since $\Aut(\V)\leq \G$, we have that $\G$ contains a translation by a non-zero vector, and for the same reason it contains all translations. Again by Theorem~\ref{thm:fundthmaffine}, we have that $\G=\AGL(\V)$.
\end{proof}

We define a relation $R\subseteq V^4$ by
\begin{align*}
    (a,b,c,d)\in R :\iff a+b=c+d.
\end{align*}
{The structure $(V, R)$ is clearly a reduct of $\V$. Every translation and every automorphism of $\V$ preserves $R$, hence 
by what we have shown, $\Aut(V,R)=\AGL(\V)$ or $\Aut(V,R)=\Sym(V)$; the latter is clearly absurd.
\begin{proposition}\label{fin_red_not_fix_0}
{Assume that a closed permutation group $\G$ on $V$ containing $\Aut(\V)$ does not fix $\0$. }Then precisely one of the following holds.
\begin{enumerate}
    \item  $\G=\Sym(V)$;
    \item  $\G=\AGL(\V)$.
\end{enumerate}
In particular there are, up to interdefinability, two first-order reducts of $\V$ which do not first-order define $\0$:
\begin{itemize}
    \item The pure set with domain $V$ and no operations or relations. 
    \item The structure $(V, R)$. \end{itemize}
\end{proposition}

Together with Proposition~~\ref{fin_red_fix_0}, Proposition~~\ref{fin_red_not_fix_0} shows that every countably infinite vector space over a finite prime field of odd order has, up to interdefinability, only finitely many first-order reducts. } This proves Theorem~\ref{thm:main}. Our proof is summarised in Figure~\ref{fig:summary}.

\newpage

\begin{figure}
\begin{small}
\begin{tikzpicture}
	\matrix[row sep=0.9cm, column sep=-0.8cm, nodes={align=center, rectangle, draw, rounded corners, thick}] {
& & & &\node (red) {The reducts of $\V \longleftrightarrow$  \\ Closed permutation groups \\ on $V$ containing $\Aut(\V)$}; & & &  \\
& & & &	\node (Gclos) {$\G$ closed, such that \\ $\Aut(\V)\leq \G\leq \Sym(V)$};  & & & \\
& & \node (GnfZ) {$\G$ does  not \\ fix $\0$}; & & & & \node (GdfZ) {$\ \ \G \ $ does \ \ \\ fix $\0$};& \\
& \node (GisE) {$\Gamma=\{1\}$}; & \node[draw=none] (P) {}; & \node (GisF) {$\Gamma=\F{p}^\times$};  & &\node (GIsF) {$\Gamma=\F{p}^\times$}; & & \node (GsmF) {$\Gamma \lneq \F{p}^\times$};\\
& & & & \node (FTPG) {Fundamental \\ Theorem of \\ Projective Geometry}; & & & \\
& & &  & & \node (GactA) {$\G$ acts on $V/_\sim$ \\as $\Aut(\V)$}; & & \node (GactS) {$\G$ acts on $V/_\sim$ \\ as $\Sym(V/_\sim)_{\{\0\}}$};\\
& & & \node [draw=none] (help1) {}; & \node (FTAG) {Fundamental \\ Theorem  of\\ Affine Geometry}; & & & \\
& \node (GisSy) {$\G=\Sym(V)$};& &\node (GisAG) {$\G=\AGL(\V)$}; & & \node (GcirA) {$G=$  \\ $G^\ast\Aut(\V)$};  & & \node (GcirS) {$G=$ \\ $G^\ast\Sym((V\setminus\{\0\})/_\sim)^f$};\\
& & & & & \node (GisSGN) {$\Gast=\SG{\Nr}{\Hc},$ \\  $\Nr\lhd \Hc \leq \Sym(\Gamma)$}; & & \node (GisSG) {$\Gast=\SG{\Nr}{\Hc},$ \\  $\Nr\lhd \Hc \leq \Sym(\Gamma)$}; \\
	};
\graph [ edges={shorten >=2pt,thick}]{
	(red) -> (Gclos) ->[to path={-| (\tikztotarget)}] 
	{ (GnfZ) ->[to path={-| (\tikztotarget)}]  
		{(GisE) -> (GisSy);
		 (GisF) -> (GisAG)};
	  (GdfZ) ->[to path={-| (\tikztotarget)}] 
	  	{(GsmF) -> (GactS) -> (GcirS) -> (GisSG);
	  	 (GIsF) -> (GactA) -> (GcirA) -> (GisSGN)};	
  	 };
    (GIsF) -> [to path={(\tikztostart) ++(0.2,-9pt)-- ++(0,-20pt) -| ($(\tikztotarget)+(-0.2,16.7pt)$)}] (GactS);
	(GactS) -- [to path={(\tikztostart) ++(-0.2,-16.5pt)-- ++(0,-7pt) -| ($(\tikztotarget)+(0,1.8pt)$)}] (P);
	(P) -> [to path={($(\tikztostart)+(-0.2,0)$)--(\tikztotarget)}](GisF);
	(P) -> (GisE);
   (FTPG) -> (FTAG);
   (FTAG)  -> [to path={-| ($(\tikztotarget) +(0.2,9.5pt)$)}] (GisAG);
   (FTPG) ->[to path={-| ($(\tikztotarget) +(-0.2,16pt)$)}] (GactA);
};
\end{tikzpicture}
\end{small}
\caption{Summary of the results.}
\label{fig:summary}
\end{figure}
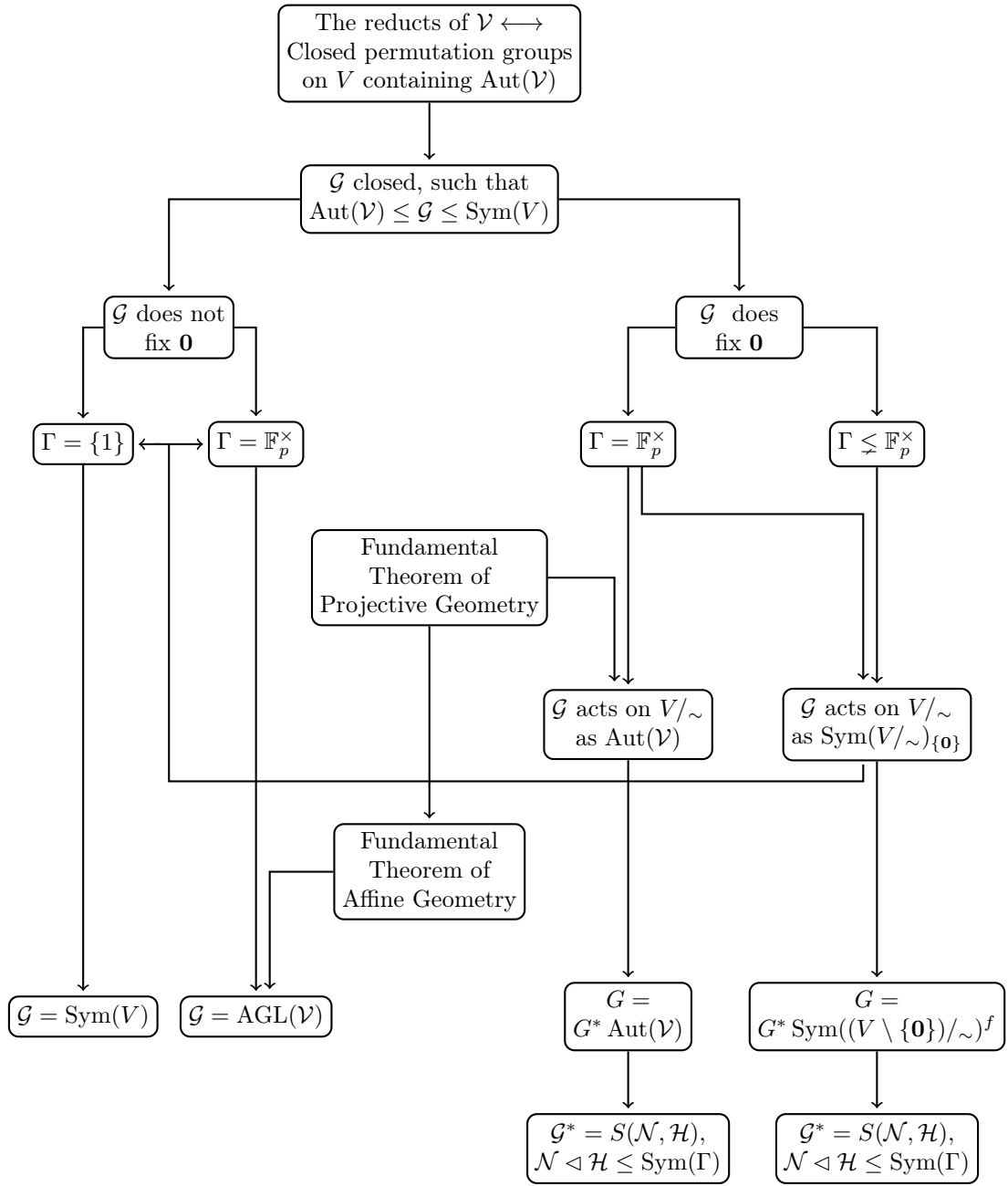

\bibliographystyle{plain}
\bibliography{global.bib,literaturverzeichnis.bib}

\end{document}